\documentclass[12pt]{amsart}

\usepackage[latin1]{inputenc}
\usepackage{amsmath}
\usepackage{amsfonts}
\usepackage{amssymb}
\usepackage{graphics}
\usepackage{enumerate}
\usepackage{subfigure}
\usepackage{amssymb,amsmath,amsthm,amscd,epsf,latexsym,verbatim,graphicx,amsfonts}
\input epsf.tex

\usepackage{amscd}
\usepackage{amsmath}
\usepackage{amssymb}
\usepackage{amsthm}
\usepackage{epsf}
\usepackage{latexsym}
\usepackage{verbatim}
\input epsf.tex

\newtheorem{theorem}{Theorem}[section]
\newtheorem{lemma}[theorem]{Lemma}
\newtheorem{corollary}[theorem]{Corollary}

\newtheorem{proposition}{Proposition}[section]

\numberwithin{equation}{section}

\newcommand{\bbZ}{\mathbb{Z}}
\newcommand{\bbR}{\mathbb{R}}
\newcommand{\bbC}{\mathbb{C}}

\newcommand{\bbD}{\mathbb{D}}
\newcommand{\bbN}{\mathbb{N}}
\newcommand{\bbH}{\mathbb{H}}
\newcommand{\eitheta}{e^{i\theta}}

\newcommand{\mcg}{\mathcal{G}}
\newcommand{\mcn}{\mathcal{N}}
\newcommand{\Real}{\textrm{Re}}

\newcommand{\mutil}{\tilde{\mu}}
\newcommand{\supp}{\textrm{supp}}

\newcommand{\cpct}{\mbox{cap}}
\newcommand{\bard}{\overline{\bbD}}
\newcommand{\barg}{\overline{G}}
\newcommand{\barc}{\overline{\bbC}}
\newcommand{\dtpi}{\frac{d\theta}{2\pi}}

\newcommand{\nri}{n\rightarrow\infty}

\begin{document}

\title[ ] {Asymptotic Properties of Extremal Polynomials Corresponding to Measures Supported on Analytic Regions}

\bibliographystyle{plain}

\thanks{  }

\maketitle

\begin{center}
\textbf{Brian Simanek}\footnote{Mathematics MC 253-37, California Institute of Technology, Pasadena, CA 91125, USA. E-mail: bsimanek@caltech.edu.   This material is based upon work supported by the National Science Foundation Graduate Research Fellowship under Grant No. DGE-1144469.}
\end{center}

\begin{abstract}
Let $G$ be a bounded region with simply connected closure $\barg$ and analytic boundary and let $\mu$ be a positive measure carried by $\barg$ together with finitely many pure points outside $G$.  We provide estimates on the norms of the monic polynomials of minimal norm in the space $L^q(\mu)$ for $q>0$.  In case the norms converge to $0$, we provide estimates on the rate of convergence, generalizing several previous results.  Our most powerful result concerns measures $\mu$ that are perturbations of measures that are absolutely continuous with respect to the push-forward of a product measure near the boundary of the unit disk.  Our results and methods also yield information about the strong asymptotics of the extremal polynomials and some information concerning Christoffel functions.
\end{abstract}

\vspace{4mm}


\vspace{2mm}


\vspace{2mm}

\normalsize

\section{Introduction}\label{intro}

\subsection{Background}\label{background}

Consider a finite and positive measure $\mu$ of compact and infinite support $\supp(\mu)\subseteq\bbC$.  Given such a measure and any $q>0$, we can define the sequence of monic polynomials $\{P_n(z;\mu,q)\}_{n=0}^{\infty}$ by letting $P_n(z;\mu,q)$ be any monic polynomial satisfying
\[
\|P_n(z;\mu,q)\|_{L^q(\mu)}=\inf\{\|Q_n\|_{L^q(\mu)}:Q_n=z^n+\mbox{ lower order terms}\};
\]
a property called the \textit{extremal property} of the polynomials $P_n(z;\mu,q)$.  For $q>1$, the strict convexity of the norm implies such a polynomial is unique, while it need not be when $0<q\leq1$ (see page 84 in \cite{StaTo}, see also Proposition \ref{nonunique} in the appendix for a treatment of the case $q=1$).  When the meaning is clear, we will often omit the $z$, $\mu$, or $q$ dependence of $P_n(z;\mu,q)$ in our notation.  By dividing each $P_n(\mu,q)$ by its $L^q(\mu)$ norm, we obtain the sequence of normalized polynomials $\{p_n(\mu,q)\}_{n=0}^{\infty}$ (we use the word ``norm" here loosely as it is not technically a norm when $q<1$).  In case $q=2$, the polynomial $p_n(\mu,q)$ is just the orthonormal polynomial for the measure $\mu$.  For an extensive introduction to the general theory of orthogonal polynomials - especially on the real line and the unit circle - we refer the reader to the references \cite{NevaiOP,OPUC1,Rice,StaTo,Suetin,Szego} and references therein.

We will consider measures whose support is contained in some compact and simply connected set $\barg$ along with finitely many points not in $\barg$.  We will also assume that $G$ is a region with analytic boundary (as defined on page 42 in \cite{GarnMar}) and that the logarithmic capacity (see section \ref{tools} below) of $G$ is equal to $1$.  One of our main tools for studying these extremal polynomials is the conformal map $\psi$ mapping the exterior of the closed unit disk $\bard$ to $\overline{\bbC}\setminus\barg$ and satisfying $\psi(\infty)=\infty$ and $\psi'(\infty)>0$.  We will denote the inverse function to $\psi$ by $\phi$.  Since $G$ has analytic boundary, the map $\psi$ can be extended to be univalent (that is, holomorphic and injective) on a slightly larger region, namely the exterior of the closed disk of radius $\tilde{\rho}$ for some $\tilde{\rho}<1$.  From now on we will assume that $\rho$ is some fixed number in the interval $(\tilde{\rho},1)$.

If we define $A_r:=\{z:r\leq|z|\leq1\}$ for every $r\in[\rho,1]$ and define $G_r:=\psi(A_r)$ then $\psi$ and $\phi$ provide a one-to-one correspondence between measures on $A_{\rho}$ and $G_{\rho}$.  Given a measure $\lambda$ on $G_{\rho}$, we will denote the corresponding measure $\kappa$ on $A_{\rho}$ by $\phi_*\lambda$.  By this we mean that for all $f\in C(\barg)$, we have
\[
\int_{G_{\rho}}f(z)d\lambda(z)=\int_{A_{\rho}}f(\psi(w))d\kappa(w).
\]
Similarly, we can write $\lambda=\psi_*\kappa$.  For example, the equilibrium measure for $\barg$ can written as $\psi_*(\dtpi)$ (see Theorem 3.1 in \cite{ToTrans}).

Central to the theory of $L^q$ extremal polynomials on a smooth Jordan curve is the analog of \textit{Szeg\H{o}'s Theorem} on the unit circle.  This can be stated as the following theorem, which follows from Theorem 7.1 in \cite{Geronimus}:

\begin{theorem}\label{geronseven}\cite{Geronimus}
If $\mu$ is a finite measure on an analytic Jordan curve $\Gamma$ having capacity $1$, then
\begin{align}\label{szegstat}
\lim_{\nri}\|P_n(\mu,q)\|^q_{L^q(\mu)}=\exp\left(\int_0^{2\pi}\log((\phi_*\mu)'(\theta)))\frac{d\theta}{2\pi}\right).
\end{align}
\end{theorem}

Any measure for which the right hand side of (\ref{szegstat}) is finite will be called a \textit{Szeg\H{o} measure} on $\partial\bbD$.  Szeg\H{o}'s Theorem can also be stated for measures on the real line (see Theorem 1.1 in \cite{FinGap2} for a precise statement).  We note here that Szeg\H{o}'s Theorem for analytic curves - as we have stated it - does not require $\mu$ to be a probability measure.

There has also been considerable research on orthogonal polynomials for measures supported on regions.  Substantial results were introduced by Carleman in \cite{Carleman} and major achievements in the field since then include the works of Ullman \cite{Ull1,Ull2}, Suetin \cite{Suetin}, Lubinsky \cite{LubBerg}, Mi\~{n}a-D\'{i}az \cite{MDPoly}, Saff \cite{SaffConj}, Stylianopoulos \cite{Corner}, Totik \cite{ToTrans}, and Widom \cite{Wid2} among others.  Recently, Nazarov, Volberg, and Yuditskii showed in \cite{Koosis} that the appropriate analog of Szeg\H{o}'s Theorem holds when $\mu$ can be written as the sum of a measure carried by $\bbD=\{z:|z|<1\}$, a Szeg\H{o} measure on $\partial\bbD$ with no singular part, and a pure point measure carried by the compliment of $\bard$.  Their result motivates our investigation in many ways.  We provide leading order asymptotics for the monic orthogonal polynomial norms in a related setting.  A special case of our main theorem (Theorem \ref{bigone} below) applies to measures similar to those considered in \cite{Koosis} although we allow for a singular component to the measure on $\partial\bbD$, but we only allow for finitely many pure points outside $\bard$.

Our results can also be motivated by the conjecture in \cite{SaffConj}.  If a measure on $\bbD$ is given by $w(z)d^2z$ where $w>0$ Lebesgue almost everywhere in some annulus with outer boundary $\partial\bbD$, the conjecture asserts that
\[
\lim_{\nri}\frac{p_{n+1}(z;\mu,2)}{zp_{n}(z;\mu,2)}=1
\]
uniformly on compact subsets of $\barc\setminus\bard$.  We will settle this conjecture in the affirmative if the weight factors as $w(r\eitheta)=h(r\eitheta)f(\theta)g(r)$ with $h$ continuous and non-vanishing on $\bard$, $f$ a Szeg\H{o} weight, and $1\in\supp(g(r)dr)$.  Indeed our main theorem considers measures that can be thought of as perturbations of such measures.  We consider several different kinds of perturbations, including adding finitely many pure points outside $\bbD$, allowing a singular component to the angular measure with weight $f(\theta)$, and the addition of a measure whose density at the boundary is negligible compared to $w(z)d^2z$.

Throughout this paper, for a measure $\gamma$ (on any set), we denote
\[
c_t(\gamma)=\int_{\bbC}|z|^t\,d\gamma(z)
\]
where we do \textit{not} insist $t\in\bbN$.  We will see that these ``moments" provide the appropriate rate of decay of the norms of the extremal polynomials.  One of our main results is the following:

\begin{theorem}\label{bigone}
Consider the measure $\mutil(r\eitheta)=h(r\eitheta)\left(\nu(\theta)\otimes\tau(r)\right)+\sigma_2(r\eitheta)$ where
\begin{enumerate}
\item $h(z)$ is a continuous function on $\bard$ that is non-vanishing in a neighborhood of $\partial\bbD$,
\item $\sigma_2$ is a measure carried by $A_{\rho}$ that satisfies $\lim_{t\rightarrow\infty}c_t(\sigma_2)c_{t}(\tau)^{-1}=0$,
\item $\nu$ is a measure on the unit circle such that $\nu'(\theta)>0$ Lebesgue almost everywhere,
\item $\tau$ is a measure on $[\rho,1]$ such that $1\in\supp(\tau)$.
\end{enumerate}
Let $\mu$ be the measure on $\bbC$ be given by
\[
\mu=\psi_*\mutil+\sigma_1+\sum_{j=1}^m\alpha_j\delta_{z_j}+\sum_{j=1}^{\ell}\beta_j\delta_{\zeta_j}
\]
where $\supp(\sigma_1)\subseteq G$, $\alpha_j,\beta_j>0$, $z_j\not\in\barg$ for all $j\in\{1,\ldots,m\}$, and $\zeta_j\in\partial G$ for all $j\in\{1,\ldots,\ell\}$.
Then
\begin{align}\label{bigconcl}
\lim_{\nri}\frac{\|P_n(z;\mu,q)\|^q_{L^q(\mu)}}{c_{qn}(\tau)}=
\exp\left(\int_0^{2\pi}\log\left(h(\eitheta)\nu'(\theta)\right)\,\dtpi\right)\prod_{j=1}^m|\phi(z_j)|^q.
\end{align}
\end{theorem}

\noindent\textit{Remark.}  We do not actually need $h$ to be continuous.  The proof of Theorem \ref{bigone} will show that $h$ need only satisfy the following conditions:
\begin{enumerate}
\item $0<\delta_1<h(r\eitheta)<\delta_2<\infty$ for all $r\in[\rho,1]$ and $\theta\in[0,2\pi]$,
\item $\int_0^{2\pi}\log(h(r\eitheta))\dtpi$ is a contunous funtion of $r\in[\rho,1]$,
\item for any polynomial $\Phi$, any $k\in\bbN$, and any $q>0$ it holds that
\[
\lim_{r\rightarrow1^-}\int_0^{2\pi}e^{ik\theta}|\Phi(\psi(r\eitheta))|^qh(r\eitheta)\dtpi=
\int_0^{2\pi}e^{ik\theta}|\Phi(\psi(\eitheta))|^qh(\eitheta)\dtpi.
\]
\end{enumerate}
For example, if $D^+=\bbD\cap\{z=x+iy:x>0\}$ and $D^-=\bbD\cap\{z=x+iy:x\leq0\}$ then $h(z)=\chi_{D^+}(z)+2\chi_{D^-}(z)$ is a function $h$ to which Theorem \ref{bigone} applies.

\vspace{2mm}

During the preparation of this manuscript, we discovered the recent results of Baratchart and Saff, which are outlined in \cite{BaraSaff}.  They consider measures on the unit disk that in may ways resemble the measures we consider in Theorem \ref{bigone}.  They obtain similar results on the asymptotic behavior of the monic orthogonal polynomial norms, though Theorem \ref{bigone} seems to be more general.

The factor of $\prod_{j=1}^m|\phi(z_j)|^q$ in (\ref{bigconcl}) is exactly what one would expect given the results of \cite{Kalicurve,Kaliarc,KaliKon,KiSi,Koosis,PanLi,PrYd}.  We will call a measure $\mu$ as in the statement of Theorem \ref{bigone} a \textit{push-forward of a product measure}.  Let us consider some examples of measures to which we can apply Theorem \ref{bigone}.

\vspace{2mm}

\noindent\textbf{Example.}  If we set $q=2$, $d\nu=\frac{d\theta}{2\pi}$, $d\tau=2rdr$, $\sigma_1=\sigma_2=0$, and $\ell=m=0$ then we are dealing with measures of the form $h(z)d^2z$ for a function $h$ continuous and non-vanishing on $\bard$.  Such measures with an added H\"{o}lder continuity assumption on $h$ were considered by Suetin in \cite{Suetin}.  Theorem \ref{bigone} recovers the leading term in the conclusion of Theorem 3.1 in \cite{Suetin}.

\vspace{2mm}

\noindent\textbf{Example.}  If we set $G=\bbD$, $\tau=\delta_1$, and $h=1$ then we recover a result similar to that of \cite{Koosis} (when $q=2$) that allows for a singular component of the measure on $\partial\bbD$, but only finitely many pure points outside $\bard$.  If we further set $\sigma_1=\sigma_2=0$ then we can recover the result from Theorem 2.2 in \cite{Kalicurve} (for any $G$ with analytic boundary).

\vspace{2mm}

\noindent\textbf{Example.}  Let us set $G=\bbD$ and $\tau=\frac{6}{\pi^2}\sum_{j=1}^{\infty}j^{-2}\delta_{1-2^{-j}}$, $d\nu=\dtpi$, $h=1$, $\ell=m=0$, and $\sigma_2=\sigma_1=0$.  If $s$ is a sufficiently large power of $2$ then
\[
c_s(\tau)=\frac{6}{\pi^2}\sum_{j=1}^{\infty}j^{-2}(1-2^{-j})^s\geq
\frac{6}{\pi^2\log_2(s)^2}\left(1-\frac{1}{s}\right)^{s}\geq\frac{C}{\log_2(s)^2}
\]
for some constant $C>0$.  Theorem \ref{bigone} implies that in this example, the extremal polynomial norms do \textit{not} decay like $O(n^{-1})$ as $\nri$.

\vspace{2mm}

\noindent\textbf{Example.}  Let us set $G=\bbD$, $\tau=(1-r)dr$, $d\nu=d\nu_{ac}$, $\ell=m=0$, and $\sigma_2=\sigma_1=0$.  In this case, we have $d\mu(z)=w(z)d^2z$ where the weight $w$ vanishes on the boundary.  Theorem \ref{bigone} still applies to this measure, and we will see below that we can still derive the asymptotics of the extremal polynomials outside $\bard$.

\vspace{2mm}

Theorem \ref{bigone} provides the asymptotic behavior of the norms of the $L^q(\mu)$-extremal polynomials for general $q\in(0,\infty)$.  We can also deduce the behavior of the extremal polynomials outside the compact set $\barg$, i.e. we can prove what is often referred to as \textit{strong asymptotics}.  If $\mu$ is of the form considered in Theorem \ref{bigone} with $\nu$ a Szeg\H{o} measure on $\partial\bbD$ then we can prove the following:
\begin{enumerate}
\item\label{szegasy} there are polynomials $\{y_n\}_{n\in\bbN}$ (depending on $P_n(\mu,q)$ and $q$) of degree $m$ and a function $S=S_q$ analytic and non-vanishing in $\overline{\bbC}\setminus\bard$ and positive at $\infty$ so that
\begin{align*}
\lim_{\nri}\frac{P_n(\psi(z);\mu,q)S(z)}{y_n(\psi(z))z^{n-m}S(\infty)}=1
\end{align*}
uniformly on compact subsets of $\overline{\bbC}\setminus\bard$,
\item\label{weakk}  the probability measures $|p_n(z;\mu,q)|^qd\mu$ converge weakly to the equilibrium measure for $\barg$ as $\nri$,
\item\label{sumo}  for any $z\in G$, we have $\sum_{n=0}^{\infty}|p_n(z;\mu,2)|^2<\infty$.
\end{enumerate}
Item (\ref{sumo}) follows from an argument based on Christoffel functions and the associated minimization problem.  We will discuss this in more detail in Section \ref{christo} and for all values of $q>0$.  The function $S$ in item (\ref{szegasy}) will be of the form given in (\ref{szegfundef}) below.  We will see that the polynomial $y_n$ in item (\ref{szegasy}) has a single zero near each $z_i$ for $i\in\{1,\ldots,m\}$ and shares all of its zeros with $P_n(z;\mu,q)$.

\subsection{Tools and Methods}\label{tools}

In an effort to fix notation and for the reader's convenience, we will now provide a brief summary of the main tools that we will use in our proofs.

In some of our proofs we will make heavy use of the Szeg\H{o} function, which we now define.  For a Szeg\H{o} measure $\gamma$ on $\partial\bbD$ with Radon-Nikodym derivative given by $\gamma'(\theta)$ we
define a function $S(z;q)$, which is analytic on $\{z:|z|>1\}$ by
\begin{align}\label{szegfundef}
S(z;q)=\exp\left(-\frac{1}{2q\pi}\int_0^{2\pi}\log(\gamma'(\theta))\frac{\eitheta+z}{\eitheta-z}d\theta\right)
\qquad,\qquad |z|>1,
\end{align}
which we will often denote by $S(z)$ if the intended value of $q$ is clear.  By the same argument as in the proof of Theorem 2.4.1(ii) in \cite{OPUC1}, we know $S(z;q)\in\bbH^q(\barc\setminus\bard)$.  We should mention that different authors have used different definitions of the Szeg\H{o} function.  The one we use was also used in \cite{Geronimus,Suetin}, while \cite{OPUC1,OPUC2,PrYd} prefer to define the Szeg\H{o} function slightly differently and with domain $\bbD$.  It follows from equation (6.5) in \cite{Geronimus} that
\[
|S(\eitheta;q)|^q=\lim_{r\searrow1}|S(r\eitheta;q)|^q=\gamma'(\theta)\qquad,\qquad a.e.\,\,\theta\in[0,2\pi].
\]

We will also need potential theoretic objects such as the equilibrium measure, logarithmic potential, and Green function of a compact set.  We refer the reader to the books \cite{GarnMar,Ransford,SaffTot} for additional background in potential theory and to \cite{SimPot,StaTo} for extensive applications of these ideas to orthogonal polynomials.  Given a finite measure $\gamma$ of compact support, we can define its \textit{logarithmic potential}
\[
U^{\gamma}(z):=\int_{\bbC}\log\frac{1}{|z-w|}d\gamma(w),
\]
though for some values of $z$, the integral may be $+\infty$.  We define the \textit{equilibrium measure} of a
compact set $K$ as the unique probability measure $\omega_K$ satisfying
\[
\int_K\int_K\log\frac{1}{|z-w|}d\omega_K(z)d\omega_K(w)=
\inf\left\{\int_K\int_K\log\frac{1}{|z-w|}d\gamma(z)d\gamma(w):\gamma(K)=1=\gamma(\bbC)\right\}
\]
provided the right hand side is finite.  In this case we call the left hand side the \textit{logarithmic energy}
of $\omega_K$ and denote it by $E(\omega_K)$.  It is always true that the support of the equilibrium measure $\omega_K$ is contained in the boundary of $K$ (see Theorem 3.7.6 in \cite{Ransford}).  We define the \textit{logarithmic capacity} of the compact set $K$ as $e^{-E(\omega_K)}$ and denote it by $\cpct(K)$.  In this paper, we will always assume that $\cpct(\barg)=1$ and consequently (with $\psi$ defined as in Section \ref{background}) $\psi'(\infty)=1$.  In this case, we can write
\begin{align*}
\psi(z)=z+\xi_0+\frac{\xi_1}{z}+\frac{\xi_2}{z^2}+\cdots\qquad,\qquad\xi_i\in\bbC.
\end{align*}
A measure $\gamma$ with compact support is called a \textit{regular measure} if
\[
\lim_{\nri}\|P_n(\gamma,2)\|^{1/n}_{L^2(\gamma)}=\cpct(\supp(\gamma)).
\]
The equilibrium measure will play an important role in Section \ref{strong}.  We will mention regularity again in Section \ref{product} and it is a major topic throughout \cite{StaTo}.


Armed with the notions of equilibrium measure and capacity, we can define the Green function with pole at infinity of a compact set $K$ (of positive capacity) as
\begin{align}\label{greedef}
g_{\barc\setminus K}(z;\infty):=-U^{\omega_K}(z)-\log(\cpct(K)).
\end{align}
It follows from Theorem 4.4.4 in \cite{Ransford} that the Green's function is conformally invariant, i.e. if $K_1$ and $K_2$ are simply connected compact sets in the plane and $\mathcal{F}$ is the conformal map that sends the compliment of $K_1$ to the compliment of $K_2$ mapping $\infty$ to itself and having positive derivative there, then
\[
g_{\barc\setminus K_1}(z;\infty)=g_{\barc\setminus K_2}(\mathcal{F}(z);\infty).
\]

\vspace{2mm}

We also include here a brief discussion of Faber polynomials (see \cite{MDFaber} for extensive background and references).  We will denote these polynomials by $F_n(z)$ and they are defined as the polynomial part of the Laurent expansion of $\phi^n(z)$ around $\infty$.  Since we are assuming $\cpct(\barg)=1$, we recover from formula (1.4) in \cite{MDFaber} the following two facts:
\begin{enumerate}
\item\label{monicfaber} $F_n(z)$ is a monic polynomial of degree $n$,
\item\label{faberdecay}  for $\rho<|z|\leq1$ we have
\[
F_n(\psi(z))=z^n+O(\tilde{\rho}^n)
\]
where the implied constant is uniformly bounded from above in the annulus considered.
\end{enumerate}
In case $G=\bbD$, we have $F_n(z)=z^n$.  Many of the proofs in \cite{Suetin} and the proof of the main theorem in \cite{Koosis} rely heavily on generalized Faber polynomials.  We will use Faber polynomials to obtain upper bounds on the $L^q$ norms of the extremal polynomials.

\vspace{2mm}

The remainder of the paper is organized as follows.  In Section \ref{product}, we prove Theorem \ref{bigone}.  One key step will be to use Faber polynomials and look at weak limits of the measures $\left\{\frac{|F_n(z)|^qd\mu}{c_{qn}(\tau)}\right\}_{n\in\bbN}$.  In Section \ref{strong} we will discuss strong asymptotics of the extremal polynomials for measures of the form considered in Theorem \ref{bigone}.  In Section \ref{christo} we will discuss Christoffel functions and their behavior on the set $\barg$, especially inside the region $G$.  A major theme throughout will be the many similarities with the theory of orthogonal polynomials on the unit circle (OPUC).  Many of our results produce interesting corollaries and we will point these out as we go.

Throughout this paper, we will let $\Gamma_r$ be the contour given by $\{\psi(z):|z|=r\}$ for $r>\tilde{\rho}$ and $\mcg_r$ will denote the region bounded by $\Gamma_r$.

\vspace{2mm}

\noindent\textbf{Acknowledgements.}  It is a pleasure to thank Barry Simon for encouraging me to pursue this line of inquiry and for much useful discussion.  I would also like to thank M. Lukic for his help with the reference \cite{Geronimus}.

\section{Push-Forward of Product Measures on the Disk}\label{product}

In this section, we will derive norm asymptotics for the extremal polynomials corresponding to measures of the form considered in Theorem \ref{bigone}.  We will use the Faber polynomials in conjunction with the extremal property to eventually derive an upper bound in the proof of Theorem \ref{bigone} and we will use subharmonicity of appropriate functions to derive a lower bound.  For the remainder of this section, we will let $q>0$ be fixed but arbitrary and we will denote $P_n(z;\mu,q)$ by $P_n(\mu)$ and $\|P_n(\mu)\|_{L^q(\mu)}$ by $\|P_n(\mu)\|_{\mu}$ when there is no possibility for confusion.  We begin with the following crude estimate:

\begin{proposition}\label{nuregreg}
If $\mu$ is as in Theorem \ref{bigone} then $\mu$ is regular.
\end{proposition}

\begin{proof}
We will in fact show that $\mu$ satisfies Widom's criterion (see Section 4.1 in \cite{StaTo}) from which regularity immediately follows by Theorem 4.1.6 in \cite{StaTo}.

For each $r\in(\rho,1]$, the equilibrium measure of the curve $\Gamma_r$ is absolutely continuous with respect to arc-length measure with continuous derivative bounded above and below by positive constants (see Theorem II.4.7 in \cite{GarnMar}; the constants are allowed to depend on $r$).  Let $C$ be a carrier of $\mu$ (i.e. $\mu(C)=\mu(\bbC)$).  Since $\nu'(\theta)>0$ Lebesgue almost everywhere, we conclude that
\[
\lambda_r(C\cap\Gamma_r)=\ell(\Gamma_r)
\]
for $\tau$ almost every $r\in(\rho,1]$ where $\lambda_r$ is arc-length measure on $\Gamma_r$ and $\ell(\Gamma_r)$ is the length of the curve $\Gamma_r$.  It follows that there is a sequence $r_n\rightarrow1$ such that $\omega_{\Gamma_{r_n}}(C)=1$ while clearly $\cpct(\Gamma_{r_n})\rightarrow1$.  This shows $\mu$ satisfies Widom's criterion.
\end{proof}

We will now begin developing the ideas necessary to prove the more refined estimate of $\|P_n(\mu)\|^q_{\mu}$ given in Theorem \ref{bigone}.  We begin with a lemma that immediately highlights the importance of Faber polynomials to our results.

\begin{lemma}\label{momentous}
Let $\mcn\subseteq\bbN$ be a subsequence such that
\[
w\mbox{-}\lim_{{\nri}\atop{n\in\mcn}}\frac{|F_n(z)|^qd\mu(z)}{a_n}=d\gamma
\]
where $\gamma$ is a measure on $\partial G$ and $\{a_n\}_{n\in\bbN}$ is a sequence of positive real numbers satisfying
$\lim_{\nri}a_na_{n+1}^{-1}=1$.  Then for any fixed $k\in\bbN$, we have
\[
w\mbox{-}\lim_{{\nri}\atop{n\in\mcn}}\frac{|F_{n-k}(z)|^qd\mu}{a_n}=d\gamma.
\]
\end{lemma}

\begin{proof}  Recall our notation $G_{\rho}=\{\psi(z):\rho\leq|z|\leq1\}$.
It is clear from our earlier discussion of Faber polynomials (specifically fact (\ref{faberdecay})) that all weak limits in question are measures on $\partial G$ and that $F_n$ has no zeros in $G_{\rho}$ for all sufficiently large $n$.  Now, let $f$ be a continuous function on $G_{\rho}$.  We have
\begin{align*}
\int_{G_{\rho}}f(z)\frac{|F_n(z)|^q}{a_n}d\mu(z)-&\int_{G_{\rho}}f(z)\frac{|F_{n-k}(z)|^q}{a_n}d\mu(z)=\\
&=\int_{G_{\rho}}f(z)\left(1-\frac{|F_{n-k}(z)|^q}{|F_n(z)|^q}\right)\frac{|F_n(z)|^q}{a_n}d\mu(z)\\
&=\int_{G_{\rho}}f(z)\left(1-\frac{|\phi(z)|^{q(n-k)}+O(\tilde{\rho}^n)}{|\phi(z)|^{qn}+O(\tilde{\rho}^n)}\right)\frac{|F_n(z)|^q}{a_n}d\mu(z)\\
&\rightarrow\int_{\partial G}f(z)\left(1-|\phi(z)|^{-qk}\right)d\gamma(z)\\
&=0
\end{align*}
since $|\phi(z)|=1$ when $z\in\partial G$.
\end{proof}

Our next lemma will identify some ideal choices for the sequence $\{a_n\}_{n\in\bbN}$ of Lemma \ref{momentous}.

\begin{lemma}\label{onemoment}
Let $\gamma$ be a probability measure on the unit interval $[0,1]$, let $c_n$ denote the $n^{th}$ moment of $\gamma$.  The following are equivalent:
\begin{enumerate}
\item\label{suppq}  $1\in\supp(\gamma)$,
\item\label{nthrot}  $\lim_{\nri}c_n^{1/n}=1$,
\item\label{ratoi}  $\lim_{\nri}c_{q(n+1)}c_{qn}^{-1}=1$.
\end{enumerate}
\end{lemma}

\begin{proof}
It is obvious that (\ref{suppq})$\Rightarrow$(\ref{nthrot}) and (\ref{ratoi})$\Rightarrow$(\ref{suppq}) so we need only prove that (\ref{nthrot})$\Rightarrow$(\ref{ratoi}).  To this end, we have
\[
\frac{c_{qn+q}}{c_{qn}}=1+\frac{\int_0^1r^{qn}(r^q-1)d\gamma(r)}{\int_0^1r^{qn}d\gamma(r)}.
\]
If $\lim_{\nri}c_n^{1/n}=1$ then the measures $\frac{r^{qn}d\gamma(r)}{\int_0^1r^{qn}d\gamma(r)}$ converge weakly to the point mass at $1$ as $\nri$, which implies the desired conclusion.
\end{proof}

Now we can prove the following lemma, which will be of critical importance in our proof of Theorem \ref{bigone}.

\begin{lemma}\label{weakboundary}
Let $\kappa$ be a measure on $\barg$ and $\gamma$ a measure on $\partial\bbD$ and let $\mcn\subseteq\bbN$ be a subsequence such that
\[
w\mbox{-}\lim_{{\nri}\atop{n\in\mcn}}\frac{|F_n(z)|^q}{a_n}d\kappa=d(\psi_*\gamma)
\]
where $\{a_n\}_{n\in\bbN}$ is as in Lemma \ref{momentous}.
Then
\[
\limsup_{{\nri}\atop{n\in\mcn}}\frac{\|P_n(\kappa)\|^q_{\kappa}}{a_n}\leq
\exp\left(\int_0^{2\pi}\log(\gamma'(\theta))\frac{d\theta}{2\pi}\right).
\]
\end{lemma}

\begin{proof}
By the extremal property, we have $\|P_n(\kappa)\|^q_{\kappa}\leq \|F_{n-k}(z)P_k(\psi_*\gamma)\|^q_{\kappa}$.  By Lemma \ref{momentous}, we can write
\[
\int_{\barg}\frac{|P_k(z;\psi_*\gamma)|^q|F_{n-k}(z)|^q}{a_n}d\kappa(z)
\rightarrow\int_{\partial G}|P_k(z;\psi_*\gamma)|^qd(\psi_*\gamma)
\]
as $\nri$ through $\mcn$.
Therefore
\[
\limsup_{{\nri}\atop{n\in\mcn}}a_{n}^{-1}\|P_n(z;\kappa)\|^q_{\kappa}
\leq \|P_k(\psi_*\gamma)\|^q_{\psi_*\gamma}
\]
for every $k>0$.  Since $k$ here is arbitrary, we can take the infimum over all $k$, which is no larger than the limit as $k$ tends to infinity.  The result now follows from Theorem \ref{geronseven}.
\end{proof}



As a second preparatory step for the proof of Theorem \ref{bigone}, we need to understand how to deal with pure points outside of $\barg$.  The following lemma is a consequence of the remark following the statement of Theorem 1 in \cite{What} (although Theorem 1 in \cite{What} is only stated for the orthonormal polynomials ($q=2$), the same proof works for the extremal polynomials in any $L^q(\mu)$ space).

\begin{lemma}\label{closetozone}\cite{What}
Let $\mu$ be a finite measure carried by $\barg\bigcup\{z_1,\ldots,z_m\}$ where $\barg$ is simply connected and each $z_i\not\in\barg$.  Then for any $\delta>0$, there is $N_\delta$ such that if $n>N_{\delta}$ then $\{u:|u-z_i|<\delta\}$ has at least one zero of $P_n(\mu)$ for each $i=1,2,\ldots,m$.
\end{lemma}

\noindent\textit{Remark.}  We will refine Lemma \ref{closetozone} later in this section (see Corollary \ref{closertozone}).

\vspace{2mm}

\noindent It is also shown in \cite{What} that pure points of $\mu$ inside $G$ need not attract zeros of $P_n(\mu)$.  The case of pure points on the boundary of $G$ is more subtle (see Section 10.13 in \cite{OPUC2} for more results on point perturbation). 

The following calculation will be useful also.

\begin{proposition}\label{psiint}
If $x\not\in G$ and $r\in[\rho,1]$ then
\[
\int_0^{2\pi}\log|\psi(r\eitheta)-x|^q\dtpi=\log|\phi(x)|^q.
\]
\end{proposition}

\begin{proof}
First, consider the case when $x\not\in\barg$.
It is clear that $\cpct(\overline{\mcg}_{r})=r$.  Define $\psi_r(z)=\psi(rz)$ on $\{z:|z|>\tilde{\rho}r^{-1}\}$.  Then we calculate
\begin{align*}
\log|\phi(x)|^q&=\int_0^{2\pi}\log|\eitheta-\phi(x)r^{-1}|^q\dtpi+q\log(r)\\
&=-qU^{\omega_{\bard}}\left(\frac{\phi(x)}{r}\right)+q\log(r)\\
&=qg_{\barc\setminus\bard}(\phi(x)r^{-1},\infty)+q\log(r)\\
&=qg_{\barc\setminus\overline{\mcg}_{r}}(x,\infty)+q\log(\cpct(\overline{\mcg}_{r}))\\
&=\int_{\Gamma_r}\log|y-x|^qd\omega_{\overline{\mcg}_{r}}(y)\\
&=\int_0^{2\pi}\log|\psi_r(\eitheta)-x|^q\dtpi.
\end{align*}
The first line follows from Example 0.5.7 in \cite{SaffTot}.  The second line is just the definition of the logarithmic potential.  The third line then follows from (\ref{greedef}) above and the fact that $\bard$ has logarithmic capacity $1$.  The fourth line then follows from the conformal invariance of the Green's function (see Section \ref{tools}).  The fifth line follows as the third did from the first.  Finally, the last line follows from the definition of equilibrium measure as given in Theorem 3.1 in \cite{ToTrans}.

The case $x\in\partial G$ follows by dominated convergence as in Example 0.5.7 in \cite{SaffTot}.
\end{proof}

Before we proceed with the proof of Theorem \ref{bigone}, we need to make an observation.  The upper bound will be obtained using the above lemmas, while for the lower bound we will invoke subharmonicity of a particular integrand.  This is simple enough when $q\geq1$ because every $H^1$ function is the Poisson integral of its boundary values (see Theorem 17.11 in \cite{Rudin}).  However, some care is required when $0<q<1$.  We simply note here that Theorem 17.11(c) in \cite{Rudin} combined with a well-known $L^q$ inequality (see page 74 in \cite{Rudin}) imply that if $f\in H^q(\barc\setminus\bard)$ then
\begin{align}\label{subharm}
\int_0^{2\pi}|f(\eitheta)|^q\dtpi\geq|f(\infty)|^q.
\end{align}
Now we are ready to prove Theorem \ref{bigone}.

\vspace{2mm}

\noindent\textit{Proof of Theorem \ref{bigone}.}
For now, let us assume that $\ell=m=0$ in our definition of $\mu$.
For any $k\in\bbN$, we have
\begin{align*}
\lim_{\nri}\int_{\barg}\frac{\phi(z)^k|F_n(z)|^q}{c_{qn}(\tau)}d\mu(z)&=
\lim_{\nri}\frac{\int_{\rho}^{1}\int_{0}^{2\pi}r^{k+qn}e^{ik\theta}h(r\eitheta)d\nu(\theta)d\tau(r)}{c_{qn}(\tau)}+
\lim_{\nri}\frac{\int_{\bard}z^{k}|z^{n}|^qd\sigma_2}{c_{qn}(\tau)}+o(1)\\
&=\int_0^{2\pi}e^{ik\theta}h(\eitheta)d\nu(\theta)=\int_{\partial G}\phi(z)^kd(\psi_*(h\nu)).
\end{align*}
It follows that the measures $\frac{|F_n(z)|^2}{c_{qn}(\tau)}d\mu$ converge weakly to $d(\psi_*(h\nu))$ as measures on $\barg$.  The upper bound in this case now follows from Lemma \ref{weakboundary}.

If we add finitely many pure points outside $G$, we get the desired upper bound by placing a single zero at each $z_i$ and $\zeta_i$.  More precisely, if we define the polynomials $y_{\infty}(z)$ and $\Upsilon_{\infty}(z)$ by
\begin{align}\label{ydef}
y_{\infty}(z)=\prod_{j=1}^m(z-z_j)\qquad,\qquad\Upsilon_{\infty}(z)=\prod_{j=1}^m(z-\zeta_j)
\end{align}
we have
\[
\|P_n(\mu)\|^q_{\mu}\leq\|y_{\infty}\Upsilon_{\infty}P_{n-m-\ell}(|y_{\infty}(z)\Upsilon_{\infty}(z)|^q\mu)\|^q_{\mu}=
\|P_{n-m-\ell}(|y_{\infty}(z)\Upsilon_{\infty}(z)|^q\mu)\|^q_{|y_{\infty}(z)\Upsilon_{\infty}(z)|^q\mu}
\]
and then proceed as in the case when $\ell=m=0$ and apply Proposition \ref{psiint}.

For the lower bound, Lemma \ref{closetozone} implies that for each $z_i$, we can choose a sequence $\{w_{i,n}\}_{n\in\bbN}$ so that $P_n(w_{i,n};\mu)=0$ and $\lim_{\nri}w_{i,n}=z_i$ (we will establish later that such a sequence has a unique tail, but we do not need this now).  Define
\begin{align}\label{whydefs}
y_n(z)=\prod_{j=1}^m(z-w_{j,n})
\end{align}
(so that $y_n(z)\rightarrow y_{\infty}(z)$ pointwise).
We now can calculate
\begin{align}\label{philow}
\|P_n(z;\mu)\|^q_{\mu}\geq
\int_{\rho}^1\int_0^{2\pi}\left|\frac{P_{n}(\psi(r\eitheta))}{y_n(\psi(r\eitheta))}\right|^q
\prod_{j=1}^m|\psi(r\eitheta)-w_{j,n}|^qh(r\eitheta)d\nu_{ac}(\theta)d\tau(r)
\end{align}
For $|z|>1$ and $r\in[\rho,1]$, define the functions
\[
S_{r,n}(z)=\exp\left(-\frac{1}{2q\pi}\int_0^{2\pi}\log\left(\prod_{j=1}^m|\psi(r\eitheta)-w_{j,n}|^q
h(r\eitheta)\nu'(\theta)\right)\frac{\eitheta+z}{\eitheta-z}d\theta\right).
\]
By our discussion in Section \ref{tools}, we can rewrite (\ref{philow}) as
\[
\|P_n(z;\mu)\|^q_{\mu}\geq
\int_{\rho}^1\int_0^{2\pi}
\left|\frac{P_{n}(\psi(r\eitheta))}{e^{i(n-m)\theta}y_n(\psi(r\eitheta))}\right|^q|S_{r,n}(\eitheta)|^q\dtpi d\tau(r)
\]
(notice that we arbitrarily added a factor of $e^{-i(n-m)\theta}$ to the integrand, which is acceptable since it is inside the absolute value bars).  For each fixed $r$, we invoke the subharmonicity of the integrand (or equation (\ref{subharm}))
to obtain
\begin{align}\label{lowermoment}
\|P_n(z;\mu)\|^q_{\mu}\geq\int_{\rho}^1r^{qn-qm}|S_{r,n}(\infty)|^q d\tau(r).
\end{align}
Since $w_{j,n}$ converges to $z_j$ as $\nri$ for each $j$ (by construction), we find that
\[
\liminf_{\nri}\frac{\|P_n(z;\mu)\|^q_{\mu}}{c_{qn}(\tau)}\geq
\exp\left(\int_0^{2\pi}\log\left(h(\eitheta)\nu'(\theta))\right)\,\dtpi\right)
\prod_{j=1}^m|\phi(z_j)|^q
\]
by Proposition \ref{psiint}.  This is the desired lower bound.
\begin{flushright}
$\Box$
\end{flushright}

\vspace{2mm}

The proof of Theorem \ref{bigone} produces several interesting corollaries.  The first of these shows that certain parts of the measure $\mu$ contribute only negligibly to the norm of the extremal polynomial.  The following corollary is reminiscent of Theorem 2.4.1(vii) in \cite{OPUC1}.

\begin{corollary}\label{singzero}
If $\mu$ is as in Theorem \ref{bigone} with $\nu$ a Szeg\H{o} measure on $\partial\bbD$ then
\begin{align*}
&\lim_{\nri}\bigg(\int_{\bard}|p_n(\psi(r\eitheta);\mu)|^q\,h(r\eitheta)d\nu_{\textrm{sing}}(\theta)d\tau(r)+
\int_{\barg}|p_n(z;\mu)|^qd\sigma_1(z)\,+\\
&\qquad\qquad\qquad+\int_{\bard}|p_n(\psi(r\eitheta);\mu)|^q\,d\sigma_2(r\eitheta)+
\sum_{j=1}^m\alpha_j|p_n(z_j;\mu)|^q+\sum_{j=1}^{\ell}\beta_j|p_n(\zeta_j;\mu)|^q\bigg)=0.
\end{align*}
\end{corollary}

\begin{proof}
Let us write $\mu=\mu^0+\mu^1$ where $\mu^0=\psi_*(h(\nu\otimes\tau))+\sum_{j=1}^m\alpha_j\delta_{z_j}$.  Then
\begin{align}\label{negl}
\frac{\|P_n(\mu)\|^q_{\mu}}{c_{qn}(\tau)}=\frac{\|P_n(\mu)\|^q_{\mu^0}}{c_{qn}(\tau)}+
\frac{\|P_n(\mu)\|^q_{\mu^1}}{c_{qn}(\tau)}.
\end{align}
The proof of Theorem \ref{bigone} shows that the left hand side of (\ref{negl}) and the first term on the right hand side of (\ref{negl}) both converge to the right hand side of (\ref{bigconcl}).  This shows that everything except $\mu^0$ contributes only negligibly to the norm of $p_n(z;\mu)$.  To show that the pure points outside $\barg$ contribute only negligibly to the norm, we keep our definition of $w_{1,n}$ from the proof of Theorem \ref{bigone} and we write $\mu^0=\mu^0_1+\alpha_1\delta_{z_1}$. We can now calculate
\begin{align*}
1&\geq
\frac{\int_{\bbC}\left|\frac{P_{n}(z;\mu)}{z-w_{1,n}}\right|^q|z-w_{1,n}|^qd\mu^0_1}{\|P_n(\mu)\|^q_{\mu}}
+\alpha_1|p_n(z_1)|^q\\
&\geq\frac{\|P_{n-1}(|z-w_{1,n}|^q\mu^0_1)\|^q_{|z-w_{1,n}|^q\mu^0_1}}{\|P_n(\mu)\|^q_{\mu}}
+\alpha_1|p_n(z_1)|^q\\
&=\frac{c_{qn}(\tau)\exp\left(\int_0^{2\pi}\log\left(h(\eitheta)\nu'(\theta)|\psi(\eitheta)-w_{1,n}|^q\right)\dtpi
\right)\prod_{j=2}^m|\phi(z_j)|^q}
{c_{qn}(\tau)\exp\left(\int_0^{2\pi}\log\left(h(\eitheta)\nu'(\theta)\right)\dtpi\right)
\prod_{j=1}^m|\phi(z_j)|^q}+\\
&\qquad+\alpha_1|p_n(z_1)|^q+o(1)\\
&=1+o(1)+\alpha_1|p_n(z_1)|^q,
\end{align*}
which implies the desired conclusion for $z_1$.  An identical proof works for each $z_j$ for $j=2,3,\ldots,m$.
\end{proof}

\noindent\textit{Remark.}  As a consequence of Corollary \ref{singzero}, we see that if $K\subseteq G$ is compact, then
\[
\int_K|p_n(z;\mu,q)|^qd\mu(z)\rightarrow0
\]
as $\nri$.

\vspace{2mm}

An additional consequence of Theorem \ref{bigone} is the following corollary, which is a refinement of Lemma \ref{closetozone}.

\begin{corollary}\label{closertozone}
Let $\mu$ be as in Theorem \ref{bigone} with $\nu$ a Szeg\H{o} measure on $\partial\bbD$.  There exists a $\delta>0$ and $N\in\bbN$ such that for all $n\geq N$, the polynomial $P_n(\mu)$ has a single zero in $\{u:|u-z_i|<\delta\}$ for each $i\leq m$.  If we denote this zero by $w_{i,n}$, then there is an $a>0$ so that $|w_{i,n}-z_i|\leq e^{-an}$ for all large $n$.
\end{corollary}

\begin{proof}
Lemma \ref{closetozone} establishes the existence of at least one zero of $P_n(\mu)$ in $\{u:|u-z_i|<\delta\}$ for all $i$ and all large $n$.  Now, fix $\epsilon>0$ (but small) and let $\{w_1,\ldots,w_{t(n)}\}$ denote the collection of zeros of $P_n(\mu)$ outside $\Gamma_{1+\epsilon}$.  

Define for $|z|>1$ the functions
\[
S_{r,n}(z)=\exp\left(-\frac{1}{2q\pi}\int_0^{2\pi}\log\left(\prod_{j=1}^{t(n)}|\psi(r\eitheta)-w_{j}|^q
h(r\eitheta)\nu'(\theta)\right)\frac{\eitheta+z}{\eitheta-z}d\theta\right).
\]
As in the proof of Theorem \ref{bigone}, we calculate
\begin{align}
\nonumber\frac{\|P_n(z;\mu)\|^q_{\mu}}{c_{qn}(\tau)}&\geq
\frac{\int_{\rho}^1r^{qn-qt(n)}|S_{r,n}(\infty)|^q d\tau(r)}{c_{qn}(\tau)}\geq
\frac{\int_{\rho}^1r^{qn-qt(n)}|S_{r,n}(\infty)|^q d\tau(r)}{c_{qn-qt(n)}(\tau)}\\
\label{attract}&=\frac{\int_{\rho}^1r^{qn-qt(n)}\exp\left(\frac{1}{2\pi}\int_0^{2\pi}\log\left(
h(r\eitheta)\nu'(\theta)\right)d\theta\right) d\tau(r)}{c_{qn-qt(n)}(\tau)}\prod_{j=1}^{t(n)}|\phi(w_j)|^q,
\end{align}
where we used Proposition \ref{psiint}.  From this expression, it follows that $n-t(n)$ tends to infinity as $\nri$, for if it did not, then since $|\phi(w_j)|>1+\epsilon$ for every $j\leq t(n)$, we would have $\|P_n(z;\mu)\|^{1/n}_{\mu}>1+\epsilon$ for all $n$ in some subsequence $\mcn\subseteq\bbN$, which violates the fact that $\cpct(\barg)=1$ and $\mu$ is regular (see Theorem III.3.1 in \cite{SaffTot}).

Since $n-t(n)\rightarrow\infty$, the first factor in (\ref{attract}) converges to $\exp\left(\frac{1}{2\pi}\int_0^{2\pi}\log\left(h(\eitheta)\nu'(\theta)\right)d\theta\right)$ as $\nri$ while the left hand side has limit given by the right hand side of (\ref{bigconcl}).  If for each $i\in\{1,\ldots,m\}$ we pick a sequence $\{w_{i,n}\}_{n\in\bbN}$ as in the proof of Theorem \ref{bigone} then the corresponding factor in the product (\ref{attract}) converges to $|\phi(z_i)|^q$ as $\nri$.  Therefore, it must be that
\[
\limsup_{\nri}\prod_{j=1, w_j\neq w_{i,n}}^{t(n)}|\phi(w_j)|^q\leq1.
\]
However, each factor in this product is larger than $(1+\epsilon)^q$.  We conclude that $t(n)=m$ for all sufficiently large $n$.  This implies $P_n(\mu)$ has a single zero near each $z_j$ for $j=1,\ldots,m$ when $n$ is sufficiently large.

The proof of the exponential attraction now proceeds exactly as in the last portion of the proof of Theorem 8.1.11 in \cite{OPUC1}.
\end{proof}

\noindent\textit{Remark.}   Corollary \ref{closertozone} tells us that the polynomial $P_n(\mu,q)$ has a single zero extremely close to $z_i$ for each $i\in\{1,\ldots,m\}$ and the remaining $n-m$ zeros are placed so as to minimize the $L^q(\mu)$ norm with respect to a varying weight.  It would be interesting to look at the measure $\mu$ on $\bbD\cup\{z_1,\ldots,z_m\}$ given by $d\mu=d^2z+\sum_{j=1}^m\delta_{z_j}$ (where $d^2z$ refers to area measure on the unit disk) and see if the results from \cite{MDPoly} continue to hold in this case, where the polynomial weight would be $y_{\infty}(z)$ (see (\ref{ydef}) above).

\vspace{2mm}

The upper bound in the proof of Theorem \ref{bigone} came from Lemma \ref{weakboundary}, which applies to arbitrary finite measures (not just product measures).  We can also state the lower bound used in the proof of Theorem \ref{bigone} in a more general form.

\begin{proposition}\label{genlow}
Let $\mutil$ be a measure on $\barg$ so that $\mutil\geq\mu$ and $\mu$ is the push-forward (via $\psi$) of the measure $w(r\eitheta)\dtpi d\tau(r)$ where $1\in\supp(\tau)$ and $w\in L^1(\dtpi\otimes d\tau(r))$.  Then
\[
\|P_n(\mutil)\|^q_{\mutil}\geq\int_0^1r^{nq}\exp\left(\int_0^{2\pi}\log(w(r\eitheta))\dtpi\right)d\tau(r).
\]
\end{proposition}

\noindent\textit{Remark.}  The statement here is very general because we do not insist on any continuity of $w$.

\begin{proof}
By the inequality of the measures and the extremal property, we have
\[
\|P_n(\mutil)\|^q_{\mutil}\geq\|P_n(\mutil)\|^q_{\mu}\geq\|P_n(\mu)\|^q_{\mu},
\]
so it suffices to put the desired bound on $\|P_n(\mu)\|^q_{\mu}$.  Let $X\subseteq[0,1]$ be the collection of all $r$ so that $w(r\eitheta)\dtpi$ is a Szeg\H{o} measure on $\partial\bbD$.  The proposition is trivial unless $\tau(X)>0$.  Therefore, we assume this is the case, and for $r\in X$ we define
\[
S_r(z)=
\exp\left(-\frac{1}{2q\pi}\int_0^{2\pi}\log\left(w(r\eitheta)\right)\frac{\eitheta+z}{\eitheta-z}d\theta\right)
\quad,\quad|z|>1
\]
and write
\[
\|P_n(\mu)\|^q_{\mu}\geq\int_Xr^{nq}|S_r(\infty)|^qd\tau(r)
\]
as in (\ref{lowermoment}).  This is the desired lower bound.
\end{proof}

We conclude this section with an example showing how one can apply Lemma \ref{weakboundary} to a region without analytic boundary.

\vspace{2mm}

\noindent\textbf{Example.}  Let $G$ be the region $\left\{z:|z^3-1|<1\right\}$ and assume $q>1$.  Notice that $G$ has capacity $1$ since $\phi(z)^3=z^3-1$ (see the example in Section 3 of \cite{MDFaber}).  Define the polynomials $Q_n$ for $n$ a multiple of $3$ by $Q_{3m}(z)=(z^3-1)^m$.

\begin{figure}[h!]\label{GDroq}
\begin{center}
\includegraphics[scale=.77]{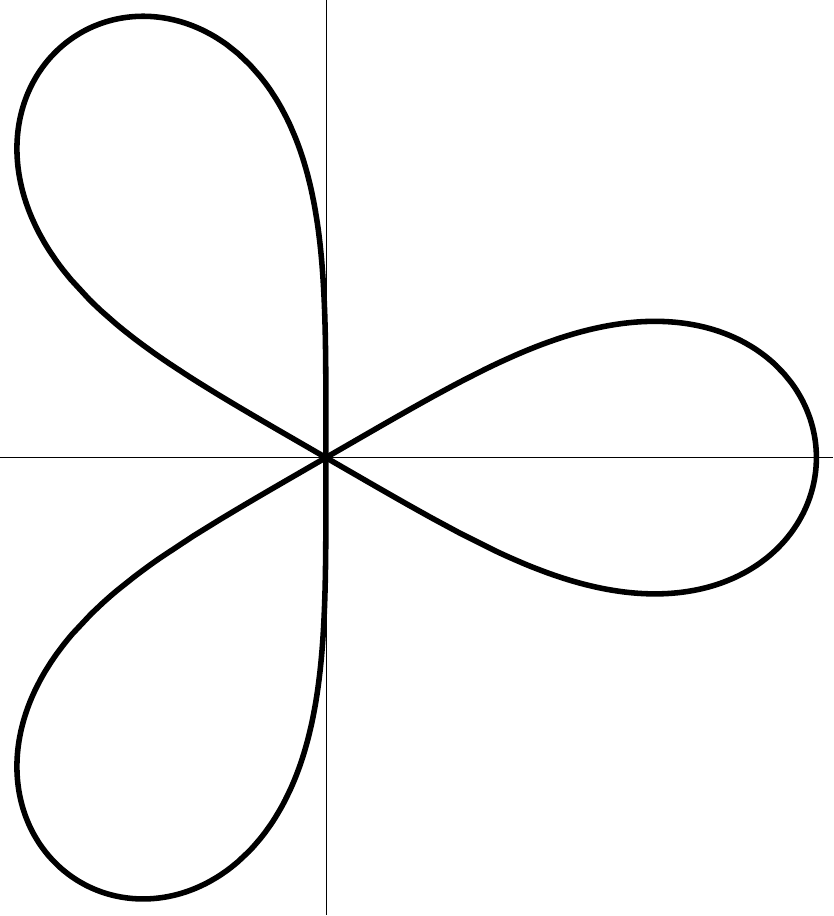}
\caption{The region $G$ of the example.}
\end{center}
\end{figure}

Let $\tau$ be a probability measure on $(0,1)$ with $1\in\supp(\tau)$.  The region $G$ can be decomposed into level sets $\Xi_r$ where
\[
\Xi_r=\left\{z:|z^3-1|=r\right\}
\]
and $r$ runs from $0$ to $1$.  Let $\nu_r$ be arc-length measure on each component of $\Xi_r$ and let $h(z)$ be a function that is continuous on $\barg$ and is invariant under rotations by $\tfrac{2\pi}{3}$ so that $\phi_*(h\nu_1)$ has $\bbZ_3$ symmetry as a measure on $\partial\bbD$ as in Example 1.6.14 in \cite{OPUC1}.  Let us define $\mu$ by
\[
\int_{\barg}f(z)d\mu(z)=\int_0^1\int_{\Xi_r}f(z)h(z)d\nu_r(z)d\tau(r).
\]
Consider the measure $h\nu_1$ on $\partial G$.  If $m\in\bbN$ is fixed, then by the extremal property we have that for any choice of complex numbers $a_0,\ldots,a_{m-1}$ and $a_m=1$
\[
\|P_{3n}(h\nu_1,q)\|_{h\nu_1}^q\leq\left\|\sum_{j=0}^ma_jQ_{3(j+n-m)}(z)\right\|_{h\nu_1}^q=\left\|\sum_{j=0}^ma_j\phi(z)^{3(j+n-m)}\right\|_{h\nu_1}^q.
\]
Following the proof of the upper bound in Theorem 7.1 in \cite{Geronimus} we get
\begin{align}\label{siv}
\|P_{3n}(h\nu_1,q)\|_{h\nu_1}^q\leq
\int_0^{2\pi}\left|1+\sum_{k=1}^m\gamma_ke^{3ki\theta}\right|^qd\phi_*(h\nu_1)
\end{align}
for any $m\leq n$ and any choice of constants $\gamma_1,\ldots,\gamma_m$.  The assumed $\bbZ_3$ symmetry of the measure implies that $P_{3m}(z;\phi_*(h\nu_1),q)=R_m(z^3)$ for some monic polynomial $R_m$ of degree $m$ (this follows from the uniqueness of the extremal polynomial in the case $q>1$; see Example 1.6.14 in \cite{OPUC1}).  Therefore, we can choose $\gamma_1,\ldots,\gamma_m$ appropriately so that the right hand side of (\ref{siv}) is equal to $\|P_{3m}(\phi_*(h\nu_1),q)\|_{\phi_*(h\nu_1)}^q$.  The reasoning of Lemma \ref{weakboundary} then implies
\[
\limsup_{\nri}\|P_{3n}(h\nu_1,q)\|_{h\nu_1}^q\leq\exp\left(\int_0^{2\pi}\log\left(\phi_*(h\nu_1)'(\theta)\right)\dtpi\right).
\]

Now, as in Lemma \ref{weakboundary}, we calculate (for $f\in C(\barg)$)
\begin{align*}
c_{qm}(\tau)^{-1}\int_{\barg}f(z)|Q_{3m}(z)|^qd\mu(z)&=
c_{qm}(\tau)^{-1}\int_{0}^1\left(\int_{\Xi_r}f(z)h(z)d\nu_r(z)\right)r^{qm}d\tau(r)\\
&\rightarrow \int_{\Xi_1}f(z)h(z)d\nu_1(z)
\end{align*}
as $m\rightarrow\infty$.  Therefore, the measures $\frac{|Q_{3m}|^q}{c_{qm}(\tau)}d\mu$ converge weakly to $hd\nu_1$ and the reasoning of Lemma \ref{weakboundary} implies
\[
\limsup_{\nri}\frac{\|P_{3n}(z;\mu,q)\|^q_{L^q(\mu)}}{c_{qn}(\tau)}\leq
\exp\left(\int_0^{2\pi}\log(\phi_*(h\nu_1)'(\theta))\dtpi\right).
\]
\begin{flushright}
$\Box$
\end{flushright}

\vspace{2mm}

In the next section, we explore more detailed asymptotic properties of the polynomials $P_n(z;\mu,q)$ and $p_n(z;\mu,q)$.

\section{Strong Asymptotics for Extremal Polynomials}\label{strong}

The main idea of Theorem \ref{bigone} is that the asymptotic behavior of the extremal polynomial norms is comparable to the behavior of the $L^q$ norms $\{\|\phi(z)^n\|_{L^q(\mu)}\}_{n\in\bbN}$.  It should not be surprising then that in some cases we can make a stronger statement about the extremal polynomials' resemblance to $\phi(z)^n$ in certain regions of the plane, and this is what we call \textit{strong asymptotics}.  Theorems \ref{szegconv} and \ref{boundarymeasure} will provide us with detailed information about the behavior of $P_n(z;\mu,q)$ outside of $\barg$ and near the boundary of $G$.  In Section \ref{christo} we will see how $P_n(z;\mu,2)$ behaves inside $G$ (see Corollary \ref{summable}).

In the previous section we established that the polynomial $P_n(\mu,q)$ has a single zero near each pure point of $\mu$ outside of $\barg$ (for large $n$) and asymptotically, all other zeros tend to $\barg$.  If we label the zero of $P_n(\mu,q)$ near $z_j$ as $w_{j,n,q}$, let us recall the definition
\begin{align*}
y_n(z;q)=\prod_{j=1}^m(z-w_{j,n,q}),
\end{align*}
which can be uniquely defined for all sufficiently large $n$ by Corollary \ref{closertozone}.  It will be convenient for us to define
\begin{align}\label{lamden}
\Lambda_n(z;\mu,q)=\frac{P_n(z;\mu,q)}{y_n(z;q)}
\end{align}
for all sufficiently large $n$.  We also recall the definition
\begin{align}\label{essen}
S_{r,n}(z;q)=\exp
\left(-\frac{1}{2q\pi}\int_0^{2\pi}\log\left(h(r\eitheta)\nu'(\theta)|y_n(\psi(r\eitheta);q)|^q\right)
\frac{\eitheta+z}{\eitheta-z}d\theta\right)
\end{align}
for $r\in[\rho,1]$ and $z\in\overline{\bbC}\setminus\bard$.  We begin by considering the behavior of $P_n(z;\mu,q)$ when $z\not\in\barg$ and $q>0$.  We will prove a result reminiscent of the convergence result in Theorem 2.4.1(iv) in \cite{OPUC1} and the corollary in \cite{Koosis}.

\begin{theorem}\label{szegconv}
Let $S_{r,n}(z;q)$ be defined as in (\ref{essen}).  If $\mu$ is a measure as in Theorem \ref{bigone} with $\nu$ a Szeg\H{o} measure on $\partial\bbD$ and $q>0$ then
\[
\frac{\Lambda_n(\psi(z);\mu,q)S_{1,\infty}(z;q)}{z^{n-m}S_{1,\infty}(\infty;q)}\rightarrow1
\]
as $\nri$ uniformly on compact subsets of $\overline{\bbC}\setminus\bard$.
\end{theorem}

\begin{proof}
Let $q>0$ be fixed throughout this proof and denote $S_{r,n}(z;q)$ by $S_{r,n}(z)$ and $\Lambda_n(z;\mu,q)$ by $\Lambda_n(z;\mu)$.

We showed in Section \ref{product} (see equation (\ref{subharm})) that if $r\in[\rho,1]$ then
\begin{align}\label{schwint}
r^{q(n-m)}S_{r,n}(\infty)^q\leq
\int_0^{2\pi}|\Lambda_n(\psi(r\eitheta);\mu)S_{r,n}(\eitheta)|^q\dtpi
\end{align}
for all $r\in[\rho,1]$.  Let us fix some $t<1$.  If we divide both sides of (\ref{schwint}) by $c_{qn}(\tau)$ and then integrate in the variable $r$ from $t$ to $1$ with respect to $\tau$, then both sides converge to $S_{1,\infty}(\infty)^q$ as $\nri$ (by Theorem \ref{bigone}).  Therefore, (\ref{schwint}) is optimal in that we cannot multiply the right hand side by a factor smaller than $1$ and have the inequality remain valid for all $r\in[t,1]$ when $n$ is sufficiently large.   It follows that for any $\epsilon>0$, there exists a sequence $\{r_n\}_{n=1}^{\infty}$ converging to $1$ from below as $\nri$ so that
\begin{align}\label{rbound}
r_n^{q(n-m)}S_{r_n,n}(\infty)^q\geq
(1-\epsilon)\int_0^{2\pi}|\Lambda_n(\psi(r_n\eitheta);\mu)S_{r_n,n}(\eitheta)|^q\dtpi.
\end{align}
By a standard argument, we can choose our sequence $\{r_n\}_{n=1}^{\infty}$ converging to $1$ from below so that (\ref{rbound}) remains true for some sequence $\epsilon_n$ tending monotonically to $0$ from above.  Let $a_n:=\|\Lambda_n(\psi(r_n\eitheta);\mu)S_{r_n,n}(\eitheta)\|_{L^q(\dtpi)}$.
By using (\ref{schwint}) and (\ref{rbound}) we see that
\begin{align}\label{infvalu}
1-\epsilon_n\leq
\left|\lim_{z\rightarrow\infty}\frac{\Lambda_n(\psi(r_nz);\mu)S_{r_n,n}(z)}{a_nz^{n-m}}\right|^q
\leq1.
\end{align}
Let
\[
f_n(z)=\frac{\Lambda_n(\psi(r_nz);\mu)S_{r_n,n}(z)}
{a_nz^{n-m}}.
\]
Clearly $\|f_n\|_{H^q(\barc\setminus\bard)}=1$ for all $n$ so we can find a subsequence $\mcn\subseteq\bbN$ so that $f_n$ converges uniformly on compact subsets of $\barc\setminus\bard$ to some (analytic) function $\tilde{f}$ (by a normal families argument and Lemma 1.1 in \cite{Kalicurve}).  Equation (\ref{infvalu}) shows that $\tilde{f}(\infty)=1$ while $\|\tilde{f}\|_{H^q(\barc\setminus\bard)}\leq\limsup_{n\in\mcn}\|f_n\|_{H^q(\barc\setminus\bard)}=1$ (see Lemma 1.2 in \cite{Kalicurve}).  However, clearly $\|\tilde{f}\|_{H^q(\barc\setminus\bard)}\geq1$ since $\tilde{f}(\infty)=1$.  This means $\int_0^{2\pi}|\tilde{f}(r\eitheta)|^q\dtpi=1$ for all $r\geq1$.  Furthermore, the proof of Corollary \ref{closertozone} and the Hurwitz Theorem imply $\tilde{f}$ is non-vanishing in $\barc\setminus\bard$ and so $f^q$ is an analytic function on this domain.  This implies
\[
\Real\left[\int_0^{2\pi}|\tilde{f}(r\eitheta)|^q-\tilde{f}(r\eitheta)^q\dtpi\right]=0\qquad,\qquad r>1
\]
and it follows easily that $\tilde{f}\equiv1$.  The same argument applies to any subsequence of $\{f_n\}_{n\in\bbN}$ so $f_n$ converges to $1$ uniformly on compact subsets of $\barc\setminus\bard$.  Equations (\ref{schwint}) and (\ref{rbound}) show that $a_n=(1+\delta_n)r_n^{(n-m)}S_{r_n,n}(\infty)$ with $\delta_n\rightarrow0$ as $\nri$.  Therefore, if $|z|>1$, we have
\begin{align}\label{exterior}
\frac{\Lambda_n(\psi(r_nz))S_{r_n,n}(z)}{(r_nz)^{(n-m)}S_{r_n,n}(\infty)}\rightarrow1
\end{align}
and the convergence is uniform on compact subsets of $\overline{\bbC}\setminus\bard$.

Now choose $w\not\in\bard$.  Notice that the function $H(z)=(\eitheta+z)(\eitheta-z)^{-1}$ is a conformal map from $\barc\setminus\bard$ to the left half-plane.  Therefore, $H(z/r_n)$ converges to $H(z)$ as $\nri$ and the convergence is uniform on compact subsets of $\barc\setminus\bard$ and uniform in $\theta$.  This observation and Dominated Convergence imply
\begin{align}\label{sconv}
\frac{S_{r_n,n}(w/r_n)S_{1,\infty}(\infty)}{S_{r_n,n}(\infty)S_{1,\infty}(w)}\rightarrow1
\end{align}
as $\nri$ and the convergence is uniform on compact subsets of $\barc\setminus\bard$.  By plugging in $z=w/r_n$ into (\ref{exterior}) and using the uniformity of convergence on compact subsets we recover
\[
\frac{\Lambda_n(\psi(w))S_{1,\infty}(w)}{w^{n-m}S_{1,\infty}(\infty)}\rightarrow1,
\]
which proves convergence.

To prove the uniformity, let $K\subseteq\barc\setminus\bard$ be a compact set.  There is a compact set $K_1\subseteq\barc\setminus\bard$ such that for all sufficiently large $n$, every $x\in K$ can be written as $r_nx_1^{(n)}$ for some $x_1^{(n)}\in K_1$.  Then
\begin{align}\label{unif}
\frac{\Lambda_n(\psi(w))S_{1,\infty}(w)}{w^{n-m}S_{1,\infty}(\infty)}&=
\frac{\Lambda_n(\psi(r_nw_1^{(n)}))S_{r_n,n}(w_1^{(n)})}{(r_nw_1^{(n)})^{n-m}S_{r_n,n}(\infty)}
\cdot\frac{S_{1,\infty}(w)S_{r_n,n}(\infty)}{S_{r_n,n}(w_1^{(n)})S_{1,\infty}(\infty)}.
\end{align}
We have already shown in (\ref{exterior}) that the first factor in (\ref{unif}) tends to $1$ uniformly on compact subsets as $\nri$.  Similarly, we have shown in (\ref{sconv}) that the second  factor in (\ref{unif}) tends to $1$ uniformly on compact subsets as $\nri$.
\end{proof}

\noindent\textit{Remark.}  One can in fact conclude that in the proof of Theorem \ref{szegconv}, the functions $f_n$ converge to $1$ in $H^q(\barc\setminus\bard)$ (see Theorem 1 in \cite{Keldysh}).

\vspace{2mm}

Theorem \ref{szegconv} yields the following Corollary, which says that the strong asymptotic behavior of the polynomials $P_n(\mu,q)$ is in some sense independent of $\tau$.

\begin{corollary}\label{stoutlim}
Let $\mu$ be as in Theorem \ref{szegconv} and let $\kappa$ be the measure on $\partial G$ given by $|y_{\infty}|^2\psi_*(h\nu)$.  Uniformly on compact subsets of $\overline{\bbC}\setminus\barg$ we have
\[
\lim_{\nri}\frac{\Lambda_n(z;\mu,q)}{P_{n-m}(z;\kappa,q)}=1.
\]
\end{corollary}

\begin{proof}
By Theorem \ref{szegconv} above and Theorem 2.1 in \cite{Kalicurve}, both $\frac{\Lambda_n(z;\mu,q)}{\phi(z)^{n-m}}$ and $\frac{P_{n-m}(z;\kappa,p)}{\phi(z)^{n-m}}$ converge to $S_{1,\infty}(\infty)S_{1,\infty}(\phi(z))^{-1}$ uniformly on compact subsets of $\overline{\bbC}\setminus\barg$ so the claim follows.
\end{proof}

Now that we have some information about the behavior of $P_n(z;\mu,q)$ outside $\barg$, we will consider what happens close to the boundary of $G$.  Our next result is motivated in part by Theorem 9.3.1 in \cite{OPUC2}.  As in Theorem \ref{szegconv}, we will consider all $q>0$.

\begin{theorem}\label{boundarymeasure}
If $\mu$ is as in Theorem \ref{bigone}, $q>0$, and $\nu$ is a Szeg\H{o} measure on $\partial\bbD$ then
\[
w\mbox{-}\lim_{\nri}|p_n(z;\mu,q)|^qd\mu(z)=d\omega_{\barg}(z)
\]
as measures on $\bbC$.
\end{theorem}

\begin{proof}
Let $q>0$ be fixed and denote by $p_n$ the polynomial $p_n(z;\mu,q)$.
Corollary \ref{singzero} and the remark following it imply that any weak limit of the measures $\{|p_n|^qd\mu\}_{n\in\bbN}$ must be a measure on $\partial G$ and that we may without loss of generality assume that $\sigma_1=\sigma_2=0$, $\ell=0$, and $\nu$ is purely absolutely continuous with respect to Lebesgue measure.  Let us recall the definition of $S_{r,n}(z)=S_{r,n}(z;q)$ from (\ref{essen}) for $r\in[\rho,1]$ and $z\in\overline{\bbC}\setminus\bard$.

We showed in Theorem \ref{bigone} that
\begin{align}\label{ratione}
\int_{\rho}^1\int_0^{2\pi}\frac{\left|e^{-i(n-m)\theta}\Lambda_n(\psi(r\eitheta))\right|^q
|S_{r,n}(\eitheta)|^q}{c_{qn}(\tau)|S_{1,n}(\infty)|^q}\dtpi d\tau(r)\rightarrow1
\end{align}
as $\nri$.

For fixed $n\in\bbN$ and $r\in[\rho,1]$, let $\{u_1,\ldots,u_{\eta_n(r)}\}$ be the zeros of $\Lambda_n(\psi(rz))$ ($\eta_n\in\bbN_0$) lying outside of $\bard$, each listed a number of times equal to its multiplicity as a zero.  We may then define the Blaschke product
\[
B_{r,n}(z)=\prod_{j=1}^{\eta_n(r)}\frac{z-u_j}{z\bar{u}_j-1}\cdot\frac{\bar{u}_j}{|u_j|}.
\]
With this notation, we may define $J_{r,n}(z)$ so that
\begin{align}\label{atinf}
z^{-(n-m)}\Lambda_n(\psi(rz))S_{r,n}(z)=B_{r,n}(z)J_{r,n}(z).
\end{align}
From (\ref{atinf}), we know that $J_{r,n}(z)$ is analytic and non-vanishing in $\barc\setminus\bard$ so we may write
\begin{align}\label{conatinf}
J_{r,n}(z)^{q/2}=J_{r,n}(\infty)^{q/2}+g_{r,n}(z)
=\left(\frac{r^{n-m}S_{r,n}(\infty)}{B_{r,n}(\infty)}\right)^{q/2}+g_{r,n}(z),
\end{align}
where $g_{r,n}(z)$ is in $\bbH^2(\barc\setminus\bard)$ and is orthogonal to the constant functions in $\bbH^2(\barc\setminus\bard)$ (that is, $g_{r,n}(z)\in\bbH^2_0(\barc\setminus\bard)$ in the notation of \cite{Duren}).  Notice that
\[
|e^{-i(n-m)\theta}\Lambda_n(\psi(r\eitheta))S_{r,n}(\eitheta)|^q=|J_{r,n}(\eitheta)^{q/2}|^2.
\]
If we plug (\ref{conatinf}) into (\ref{ratione}), we get
\begin{align}\label{gneg}
\int_{\rho}^1\frac{r^{qn-qm}|S_{r,n}(\infty)|^q}{c_{qn}(\tau)|S_{1,n}(\infty)|^qB_{r,n}(\infty)^q}+
\frac{\|g_{r,n}\|^2_{\bbH^2}}{c_{qn}(\tau)|S_{1,n}(\infty)|^q}d\tau(r)
\rightarrow1
\end{align}
as $\nri$.  However, $B_{r,n}(\infty)^{-q}>1$ and the second term is always non-negative, so we conclude that the first term in (\ref{gneg}) has integral tending to $1$ as $\nri$ and hence
\begin{align}\label{tozero}
\int_{\rho}^1\frac{\|g_{r,n}\|^2_{\bbH^2}}{c_{qn}(\tau)|S_{1,n}(\infty)|^q}d\tau(r)\rightarrow0
\end{align}
as $\nri$.

Now fix $k\in\bbN$.  We have
\begin{align}
&\label{momint} \int_{G_{\rho}}\phi(z)^k|p_n(z)|^qd\mu(z)\\
\nonumber&\qquad=\int_{\rho}^1\int_0^{2\pi}
\frac{r^ke^{ik\theta}|e^{-i(n-m)\theta}\Lambda_n(\psi(r\eitheta))|^q|S_{r,n}(\eitheta)|^q}
{c_{qn}(\tau)|S_{1,n}(\infty)|^q}\dtpi d\tau(r)+o(1)\\
\nonumber&\qquad=\int_{\rho}^1\int_0^{2\pi}\frac{r^ke^{ik\theta}|r^{q(n-m)/2}S_{r,n}(\infty)^{q/2}
B_{r,n}(\infty)^{-q/2}+g_{r,n}(\eitheta)|^2}
{c_{qn}(\tau)|S_{1,n}(\infty)|^q} \dtpi d\tau(r)+o(1)\\
&\label{htwoproof}\qquad=\int_{\rho}^1\int_0^{2\pi}\frac{r^{k+q(n-m)}e^{ik\theta}|S_{r,n}(\infty)|^q}
{c_{qn}(\tau)|S_{1,n}(\infty)|^qB_{r,n}(\infty)^{q}}\dtpi d\tau(r)
+\int_{\rho}^1\int_0^{2\pi}\frac{r^{k}e^{ik\theta}|g_{r,n}(\eitheta)|^2}
{c_{qn}(\tau)|S_{1,n}(\infty)|^q} \dtpi d\tau(r)\\
\nonumber&\qquad\qquad+\int_{\rho}^1\int_0^{2\pi}\frac{r^{k+q(n-m)/2}e^{ik\theta}S_{r,n}(\infty)^{q/2}
\cdot2\Real[g_{r,n}(\eitheta)]}{c_{qn}(\tau)|S_{1,n}(\infty)|^qB_{r,n}(\infty)^{q/2}} \dtpi d\tau(r)+o(1)
\end{align}
as $\nri$.  If we send $n$ to infinity, the first term in (\ref{htwoproof}) converges to $0$ since $k\in\bbN$.  The second term in (\ref{htwoproof}) can be bounded from above in absolute value by
\begin{align}\label{twothreebound}
\int_{\rho}^1\frac{\|g_{r,n}\|^2_{\bbH^2}}{c_{qn}(\tau)|S_{1,n}(\infty)|^q}d\tau(r),
\end{align}
which tends to $0$ by (\ref{tozero}).  By applying the Schwartz inequality, the third term in (\ref{htwoproof}) can be bounded from above in absolute value by
\[
\left(\int_{\rho}^1\frac{r^{2k+q(n-m)}S_{r,n}(\infty)^q}{c_{qn}(\tau)S_{1,n}(\infty)^qB_{r,n}(\infty)^q}d\tau(r)\right)^{1/2}
\left(\int_{\rho}^1\frac{4|\int_0^{2\pi}e^{ik\theta}\Real[g_{r,n}(\eitheta)]\dtpi|^2}
{c_{qn}(\tau)S_{1,n}(\infty)^q}d\tau(r)\right)^{1/2}
\]
The first factor tends to $1$ as $\nri$ (as in (\ref{gneg})).  After applying Jensen's inequality to the second factor, we can bound it from above by twice the square root of (\ref{twothreebound}).  Therefore the integral (\ref{momint}) tends to $0$ as $\nri$.

We conclude that if $\gamma$ is a weak limit point of the measures $\{|p_n(\mu)|^qd\mu\}_{n\in\bbN}$, then for every $k\in\bbN$ we have
\[
\int_{\partial G}\phi(z)^kd\gamma=0.
\]
This implies that $\gamma$ is induced (via $\psi$) by a measure $\kappa$ on $\partial\bbD$ with no non-trivial moments, i.e. $d\kappa=\dtpi$ and it follows that $\gamma$ is the equilibrium measure for $\barg$ (see Theorem 3.1 in \cite{ToTrans}).
\end{proof}

Theorem \ref{boundarymeasure} yields the following corollary, which can be interpreted in terms of the Christoffel functions discussed in Section \ref{christo} (see (\ref{lambdn}) below).

\begin{corollary}\label{kernelmeasure}
Under the hypotheses of Theorem \ref{boundarymeasure}, we have
\[
w\mbox{-}\lim_{\nri}\frac{K_n(z,z)}{n+1}d\mu(z)=d\omega_{\barg}
\]
as measures on $\bbC$ where $K_n(z,w)=\sum_{j=0}^np_j(z;\mu,2)\overline{p_j(w;\mu,2)}$ is the reproducing kernel for the measure $\mu$ and polynomials of degree at most $n$.
\end{corollary}

\noindent\textit{Remark.}  Since $\mu$ is regular, one can use a polynomial approximation argument, Corollary \ref{singzero}, and the results in \cite{WeakCD} to arrive at a different proof of Corollary \ref{kernelmeasure}.  Theorem \ref{boundarymeasure} is of course much stronger.

\vspace{2mm}

In the next section, we will consider the behavior of the Christoffel functions on $\barg$.

\section{Christoffel Functions}\label{christo}

In this section we will turn our attention to an interesting minimization problem.  For each $n\in\bbN$ and $q>0$, let us define the \textit{Christoffel function} $\lambda_n(z;\mu,q)$ by
\begin{align*}
\lambda_n(z;\mu,q)=\inf\left\{\int_{\bbC}|Q(w)|^qd\mu(w):Q(z)=1\,,\,\textrm{deg}(Q)\leq n\right\}.
\end{align*}
For $z\in\bbC$ fixed, $\lambda_n(z;\mu,q)$ is obviously non-increasing (as $\nri$) and positive, so we may define $\lambda_{\infty}(z;\mu,q)=\lim_{\nri}\lambda_n(z;\mu,q)$.  It is clear that
\[
\lambda_{\infty}(z;\mu,q)=\inf\left\{\int_{\bbC}|Q(w)|^q\,d\mu(w):
\, Q(z)=1\, ,\,Q \mbox{ polynomial}\right\}.
\]
The behavior of $\lambda_{\infty}(z;\mu,q)$ is particularly easy to describe when $z\in\partial G$.

\begin{proposition}\label{christoboundary}
If $\mu$ is any measure with support in $\barg$ and $G$ has analytic boundary then $\lambda_{\infty}(x;\mu,q)=\mu(\{x\})$ for all $x\in\partial G$ and all $q>0$.
\end{proposition}

\noindent\textit{Remark.}  For Proposition \ref{christoboundary}, we do not need to assume $\cpct(\barg)=1$.

\begin{proof}
Fix $x\in\partial G$.  It is obvious that $\lambda_n(x;\mu,q)\geq\mu(\{x\})$ for every $n\in\bbN$, so it remains to show the reverse inequality holds in the limit.  Since $\partial G$ is analytic, we can define a conformal map $\varphi:G\rightarrow\bbD$ satisfying $\varphi(x)=1$.  By a well-known argument, this map $\varphi$ has an analytic continuation to some open set $U\supseteq\barg$.  Define
\[
f_n(z):=3^{-n}(\varphi(z)+2)^n\qquad,\qquad z\in U
\]
so that $f_n(x)=1=\|f_n\|_{L^{\infty}(\barg)}$.  By Theorem 2.5.7 in \cite{StSh} there exists a sequence of polynomials $\{W_n\}_{n\in\bbN}$ so that $\|W_n-f_n\|_{L^{\infty}(\barg)}<n^{-1}$ (we do not assume $W_n$ has degree $n$).  It follows that for each $n\in\bbN$ there is a constant $a_n=1+o(1)$ (as $\nri$) so that $a_nW_n(x)=1$.  Then (with $E_n=W_n-f_n$)
\[
\lambda_n(x;\mu,q)\leq\int_{\barg}|a_nW_n(z)|^qd\mu(z)=(1+o(1))\int_{\barg}|f_n(z)+E_n|^qd\mu(z)
\rightarrow\mu(\{x\})
\]
by Dominated Convergence.
\end{proof}

\noindent\textit{Remark.}  For results producing more precise asymptotics of $\lambda_n(z;\mu,2)$ for $z\in\partial G$ under stronger hypotheses on $\mu$, see \cite{LubBerg,ToTrans}.

\vspace{2mm}

Now let us focus on $x\in G$.  For measures supported on the unit circle, it is known (see Theorem 2.5.4 in \cite{OPUC1}) that if $\nu$ is a Szeg\H{o} measure then $\lambda_{\infty}(z;\nu,q)>0$ for all $z\in\bbD$ and $q\in(0,\infty)$.  We will prove an analog for the kinds of measures considered in Theorem \ref{bigone}.  Before we can do this, we need to define some auxiliary notation.  For $x$ interior to $\Gamma_{1}$, define
\[
\xi(x)=\frac{1}{2}\left(1+\inf\{r:x\in\mcg_r\,,\, r\geq\rho\}\right).
\]
For each $r\in[\xi(x),1]$, let $\chi_{r,x}$ be the conformal map from $\bbD$ to $\mcg_r$ that sends $0$ to $x$ and satisfies $\chi_{r,x}'(0)>0$.  Denote the inverse to $\chi_{r,x}$ by $\varphi_{r,x}$.  The following lemma will be useful:

\begin{lemma}\label{varphideriv}
With the above notation, it holds that $\varphi_{r,x}$ converges to $\varphi_{1,x}$ uniformly on some open set containing $\barg$ as $r\rightarrow1$ and there is an $s\in(\xi(x),1)$ and positive constants $\lambda_1$ and $\lambda_2$ 
such that
\[
\lambda_1<|\varphi_{r,x}'(z)|<\lambda_2
\]
for all $r\in[s,1]$ and $z\in\barg$.
\end{lemma}

\noindent\textit{Remark.}  The proof of the lemma will actually show that when $r$ is sufficiently close to $1$, $\varphi_{r,x}$ is defined on all of $\barg$ so the statement of the lemma makes sense.

\begin{proof}
By the Carath\'{e}odory Convergence Theorem (see Theorem 3.1 in \cite{Dur}), the maps $\varphi_{r,x}$ converge to $\varphi_{1,x}$ uniformly on compact subsets of $G$ as $r\rightarrow1^-$ (see also Theorem 3 in \cite{SnipWard}).  Since $G$ has analytic boundary, a simple argument shows that each $\varphi_{r,x}$ can be univalently continued outside of $\barg$ when $r$ is sufficiently close to $1$ and in fact all such $\varphi_{r,x}$ have a common domain of holomorphy containing $\barg$.  A normal families argument then implies $\varphi_{r,x}$ converges to $\varphi_{1,x}$ uniformly on some open set containing $\barg$ as $r\rightarrow1$.  We can then use the Cauchy integral formula to conclude that $\varphi_{r,x}'$ converges to $\varphi_{1,x}'$ on a smaller open set containing $\barg$.  This means that when $r$ is sufficiently close to $1$, we have $\|\varphi_{r,x}'\|_{L^{\infty}(\Gamma_1)}\leq2\|\varphi_{1,x}'\|_{L^{\infty}(\Gamma_1)}$.  The same arguments can be applied to $\{\chi_{r,x}\}_{r\in[\xi(x),1]}$, which proves the claim.
\end{proof}

As a final preparatory step, we will need the following lemma, which is a slight refinement of Lemma 1.1 in \cite{Kalicurve}.

\begin{lemma}\label{hpbound}
If $q\in(0,\infty)$ and $w\in\mcg_r$  then there is a constant $\beta_w$ so that for every $f\in H^q(\mcg_r)$,
\[
|f(w)|^q\leq\beta_w\int_{\Gamma_r}|f(z)|^qd|z|.
\]
Furthermore, the constant $\beta_w$ may be taken uniform for all $r$ sufficiently close to $1$ (but perhaps depending on $w$).
\end{lemma}

\begin{proof}
The inequality follows from Lemma 1.1 in \cite{Kalicurve} and the equivalence of the spaces $E^q(\mcg_r)$ and $H^q(\mcg_r)$ (see Chapter 10 in \cite{Duren}), so we need only focus on the uniformity.  If $q\geq1$, then this is a simple consequence of Jensen's inequality and the fact that $H^1$ functions are the Cauchy integral of their boundary values (see Theorem 10.4 in \cite{Duren}), so we need only focus on the case $0<q<1$.  To this end, let $g$ be the function harmonic in $\mcg_r$ satisfying $g(\psi(r\eitheta))=|f(\psi(r\eitheta))|^q$ almost everywhere on $\Gamma_r$.  Let $\omega_{r,w}$ by the harmonic measure for the region $\mcg_r$ and the point $w$.  Then by the subharmonicity of $f$, we have
\begin{align*}
|f(w)|^q\leq g(w)=\int_{\Gamma_r}g(z)d\omega_{r,w}(z)&\leq
\left\|\frac{d\omega_{r,w}}{d|z|}\right\|_{L^{\infty}(\Gamma_r)}\int_{\Gamma_r}g(z)d|z|\\
&=\left\|\varphi_{r,w}'\right\|_{L^{\infty}(\Gamma_r)}\int_{\Gamma_r}|f(z)|^qd|z|.
\end{align*}
We can now apply Lemma \ref{varphideriv} with $x=w$ to provide uniformity in the constant $\beta_w$.
\end{proof}

Now we are ready to prove the main theorem of this section.

\begin{theorem}\label{pbigo}
If $\mu$ and $G$ are as in Theorem \ref{bigone} with $\nu$ a Szeg\H{o} measure on $\partial\bbD$, then $\lambda_{\infty}(z;\mu,q)>0$ for all $z\in G$ and $q\in(0,\infty)$.
\end{theorem}

\begin{proof}
Since $h$ is bounded from below and $\lambda_n(z;\mu,q)$ increases as we increase $\mu$, we may assume that $\mu=\nu_{ac}\otimes\tau$.  In the region $G_{\rho}$ we may write (for $f$ continuous)
\begin{align}\label{moveto}
\int_{G_{\rho}}f(z)d\mu(z)=
\int_{\rho}^1\int_{\Gamma_t}f(z)\tilde{w}(z)d|z|
d\tau(t)
\end{align}
where $\tilde{w}$ is a weight on $G_{\rho}$. In fact, we can write explicitly
\begin{align}\label{proofofszego}
\tilde{w}(z)=\frac{1}{2\pi}\cdot\nu'\left(\frac{\phi(z)}{|\phi(z)|}\right)\frac{|\phi'(z)|}{|\phi(z)|}
\end{align}
(we identify $\nu'(\eitheta)$ and $\nu'(\theta)$).  As in \cite{MDCurve}, define $\Delta_{r,q}(z)$ by
\begin{align}\label{szegrpdef}
\Delta_{r,q}(z)=\exp\left(\frac{1}{2q\pi i}\oint_{\Gamma_r}\log\left(\tilde{w}(\zeta)\right)\frac{1+\overline{\varphi_r(\zeta)}
\varphi_r(z)}{\varphi_r(\zeta)-\varphi_r(z)}\varphi_r'(\zeta)d\zeta\right)
\end{align}
for each $r\in[\rho,1]$ so that $|\Delta_{r,q}(\zeta))|^q=\tilde{w}(\zeta)$ for almost every $\zeta\in\Gamma_r$ ((\ref{proofofszego}) implies the integral in (\ref{szegrpdef}) converges).

Now fix $y\in G$ and let $Q(z)$ be any polynomial so that $Q(y)=1$ (we make no assumptions on the degree of $Q$).  Let $s\in(\rho,1)$ be so that $y$ is interior to $\Gamma_s$ and so the constant $\beta_y$ of Lemma \ref{hpbound} may be chosen independently of $t\in[s,1]$.  We calculate
\begin{align}\label{pdelta}
\|Q\|^q_{L^q(\mu)}\geq\int_{s}^1\int_{\Gamma_t}|Q(z)\Delta_{t,q}(z)|^q\,d|z|d\tau(t)
\geq \beta_y^{-1}\int_s^1|\Delta_{t,q}(y)|^qd\tau(t)
\end{align}
by Lemma \ref{hpbound}.
The function $\Delta_{t,q}(y)$ is expressed as an exponential so the fact that $\nu$ is a Szeg\H{o} measure on $\partial\bbD$ implies $\Delta_{t,q}(y)$ is never equal to $0$ for any $t$.  Therefore, $|\Delta_{t,q}(y)|^q$ is not the zero function and so the integral on the far right on (\ref{pdelta}) is not equal to zero.  We have therefore obtained a lower bound for the far left hand side of (\ref{pdelta}) that is independent of the degree of $Q$.  Taking the infimum over all such $Q$ proves the theorem.
\end{proof}


Recall the definition of $K_n(z,w)$ from Corollary \ref{kernelmeasure}.  By equation (2.16.6) in \cite{Rice}, one has
\begin{align}\label{lambdn}
\lambda_n(z;\mu,2)=\frac{1}{K_n(z,z;\mu)}.
\end{align}
This and Theorem \ref{pbigo} for the case $q=2$ yields a proof of the following corollary:

\begin{corollary}\label{summable}
If $\tilde{\mu}\geq\mu$ and $\mu$ is as in Theorem \ref{bigone} with $\nu$ a Szeg\H{o} measure on $\partial\bbD$ then
\[
\sum_{n=0}^{\infty}|p_n(z;\tilde{\mu},2)|^2<\infty
\]
for all $z\in G$.
\end{corollary}


Now that we have some understanding of $\lambda_{\infty}(x;\mu,q)$ for all $x\in G$ when $\mu$ is of the form considered in Theorem \ref{bigone}, we want to try to calculate it exactly.  Our next result will show that one can reduce the problem to considering only measures on $G=\bbD$ and only the point $x=0$.  Indeed, take any $x_0\in G$ and let $\varphi$ be the conformal map of $G$ to $\bbD$ sending $x_0$ to $0$ and satisfying $\varphi'(x_0)>0$.  By the injectivity of $\varphi$ on $\barg$ (we used Carath\'{e}odory's Theorem here; see Theorem I.3.1 in \cite{GarnMar}), we can push any measure $\mu$ on $\barg$ forward via $\varphi$ to get a measure $\varphi_*\mu$ on $\bard$ as in Section \ref{intro}.  With this notation, we can prove the following result:

\begin{proposition}\label{into}
With $x_0$, $\mu$ and $\varphi$ as above, we have $\lambda_{\infty}(x_0;\mu,q)=\lambda_{\infty}(0;\varphi_*\mu,q)$ for all $q\in(0,\infty)$.
\end{proposition}

\noindent\textit{Remark.}  We do not exclude the possibility that $G=\bbD$ and $\varphi$ is an automorphism of the disk.

\vspace{2mm}

\noindent\textit{Remark.}  If $\tau\neq\delta_1$, the resulting measure $\varphi_*\mu$ may not be of the form considered in Theorem \ref{bigone}.

\vspace{2mm}

\begin{proof}
Fix $q\in(0,\infty)$.  Given $\epsilon>0$, let $T$ be a polynomial so that $\|T\|^q_{L^q(\varphi_*\mu)}<\lambda_{\infty}(0;\varphi_*\mu,q)+\epsilon$ and $T(0)=1$.  Then $\tilde{Q}:=T\circ\varphi$ is a function on $\barg$ satisfying $\|\tilde{Q}\|^q_{L^q(\mu)}=\|T\|^q_{L^q(\varphi_*\mu)}$ and $\tilde{Q}(x_0)=1$.  Now let $Q$ be a polynomial satisfying $\||Q|^q-|\tilde{Q}|^q\|_{L^{\infty}(\barg)}<\epsilon$ and $Q(x_0)=1$ (such a $Q$ exists by the same reasoning as in the proof of Proposition \ref{christoboundary}).  It follows at once that $\lambda_{\infty}(x_0;\mu,q)\leq\lambda_{\infty}(0;\varphi_*\mu,q)+2\epsilon$ and one direction of the inequality follows by sending $\epsilon\rightarrow0$.  The reverse inequality follows by an argument symmetric to the one just given.
\end{proof}

\noindent\textit{Remark.}  If we set $\tau=\delta_1$, Proposition \ref{into} can be used to provide a new proof of Proposition 2.2.2 in \cite{OPUC1} and a new proof of Theorem 2.5.4 in \cite{OPUC1}.

\vspace{2mm}

Proposition \ref{into} allows us to calculate $\lambda_{\infty}(x;\mu,q)$ by considering only measures on $\bard$ and only the point $0$.  If $\mu$ happens to be supported on $\partial G$, then $\varphi_*\mu$ is supported on $\partial\bbD$ so that $\lambda_{\infty}(0;\varphi_*\mu,q)$ is in fact independent of $q$ (see Theorem 2.5.4 in \cite{OPUC1}) so the same must be true of $\lambda_{\infty}(x;\mu,q)$.  However, the following example shows that the value of $\lambda_{\infty}(0;\mu,q)$ is in general not as easily calculated when $\supp(\mu)\nsubseteq\partial G$.

\vspace{2mm}

\noindent\textbf{Example.}  Let us consider the special case of Corollary \ref{summable} where $G=\bbD$, $h=1$, and $z=0$.  Let us further assume $\tau$ and $\nu$ are both probability measures.  Fix any $N\in\bbN$ and let $Q_N(z)$ be a polynomial of degree at most $N$ satisfying $Q_N(0)=1$. Then for any $r<1$ we have
\[
\int_0^{2\pi}|Q_N(r\eitheta)|^2d\nu(\theta)\geq\lambda_N(0;\nu,2)
\]
because $Q_N(rz)$ is still a polynomial of degree $N$ in $z$ that is equal to $1$ at $0$.  Integrating both sides in the variable $r$ with respect to $\tau$ from $0$ to $1$, we obtain $\lambda_N(0;\mu,2)\geq\lambda_N(0;\nu,2)$.  Sending $N\rightarrow\infty$ we obtain $\lambda_{\infty}(0;\mu,2)\geq\lambda_{\infty}(0;\nu,2)>0$ (see equation (2.2.3) in \cite{OPUC1}).

However, if $0\in\supp(\tau)$ then the reverse inequality is false unless $d\nu=\dtpi$ (we still assume $\nu$ is a Szeg\H{o} measure on $\partial\bbD$), i.e. it is true that $\lambda_{\infty}(0;\mu,2)>\lambda_{\infty}(0;\nu,2)$.  To see this, recall Proposition 2.16.2 in \cite{Rice}, which tells us that $Q_{n,z}(w):=K_n(z,w;\mu)K_n(z,z;\mu)^{-1}$ satisfies $Q_{n,z}(z)=1$ and $\|Q_{n,z}(z)\|^2_{\mu}=\lambda_n(z;\mu,2)$.  Corollary \ref{summable} tells us that $\{Q_{n,0}(w)\}_{n\in\bbN}$ is uniformly bounded on $\{z:|z|\leq r_1\}$ for any $r_1<1$.  By Montel's Theorem this is a normal family so we may take $\nri$ through some subsequence $\mcn\subseteq\bbN$ so that $\{Q_{n,0}(w)\}_{n\in\mcn}$ converges uniformly to a function $Q_{\infty,0}(w)$, which is analytic in $\{z:|z|<r_1\}$ and $Q_{\infty,0}(0)=1$.  By continuity and the fact that if $d\nu\neq\dtpi$ then $\lambda_{\infty}(0;\nu,2)<1$, it must be that
\[
\int_0^{2\pi}|Q_{\infty,0}(r\eitheta)|^2d\nu(\theta)>\frac{1+\lambda_{\infty}(0;\nu,2)}{2}
\]
for all $r$ sufficiently small (say $r<r_0$).  By Dominated Convergence, the same must be true for all $Q_{n,0}(z)$ for $n$ sufficiently large and $n\in\mcn$.  We conclude that for sufficiently large $n\in\mcn$, we have
\begin{align*}
\lambda_n(0;\mu,2)&=\|Q_{n,0}(z)\|^2_{\mu}=\int_0^{r_0}\int_0^{2\pi}|Q_{n,0}(r\eitheta)|^2d\nu(\theta)d\tau(r)+
\int_{r_0}^{1}\int_0^{2\pi}|Q_{n,0}(r\eitheta)|^2d\nu(\theta)d\tau(r)\\
&>\frac{1+\lambda_{\infty}(0;\nu,2)}{2}\tau([0,r_0])+\lambda_{\infty}(0;\nu,2)\tau((r_0,1])\\
&=\frac{1-\lambda_{\infty}(0;\nu,2)}{2}\tau([0,r_0])+\lambda_{\infty}(0;\nu,2).
\end{align*}
Since $\lambda_n(0;\mu,2)$ is decreasing in $n$, $\tau([0,r_0])>0$ and $\lambda_{\infty}(0;\nu,2)<1$, the desired conclusion follows.
\begin{flushright}
$\Box$
\end{flushright}

\vspace{4mm}

\section{Appendix}\label{appendix}

\subsection{Non-uniqueness when $q=1$.}\label{pequal}

On page 84 in \cite{StaTo}, it is stated that one does not have uniqueness of the $L^q$-extremal polynomial when $0<q<1$.  This is a correct statement, but we show here that this can be extended to include the case $q=1$.

\begin{proposition}\label{nonunique}
If $\mu$ is a finite measure supported on $[-2,-1]\cup[1,2]$ and $\mu(A)=\mu(-A)$ for all measurable sets $A$, then one does not have uniqueness of the $L^1$-extremal polynomial $P_n(\mu,1)$ for every odd $n$.
\end{proposition}

\begin{proof}
Suppose for contradiction that $P_{2n+1}(\mu,1)$ can be uniquely defined.  By the symmetry of the measure, we must have that $P_{2n+1}(0;\mu,1)=0$.  We may then write $P_{2n+1}(z;\mu,1)=zQ_n(z)$, for some polynomial $Q_n$ of degree $2n$ and satisfying $Q_n(x)=Q_n(-x)$ for all $x\in\bbR$.  For $a\in(-1,1)$, define
\[
P_{2n+1}^{(a)}(z)=(z-a)Q_n(z)
\]
so that $P_{2n+1}^{(0)}(z)=P_{2n+1}(z;\mu,1)$.  We then have
\begin{align*}
\frac{\partial}{\partial a}\|P_{2n+1}^{(a)}\|_{L^1(\mu)}&=\frac{\partial}{\partial a}\left(
\int_{-2}^{-1}(a-z)|Q_n(z)|d\mu(z)+\int_{1}^{2}(z-a)|Q_n(z)|d\mu(z)\right)\\
&=\int_{-2}^{-1}|Q_n(z)|d\mu(z)-\int_{1}^{2}|Q_n(z)|d\mu(z)=0,
\end{align*}
which contradicts our uniqueness assumption.
\end{proof}

If in Proposition \ref{nonunique} we also assume $\mu$ has no pure points then an alternative proof can be found by appealing to Theorem 2.1 in \cite{Pinkus}.



\vspace{10mm}




\begin{thebibliography}{14}


\bibitem{Keldysh}  M. Bello-Hern\'{a}ndez, F. Marcell\'{a}n, and J. M\'{i}nguez-Ceniceros, {\em Pseudo-uniform convexity in $H^p$ and some extremal problems on Sobolev spaces}, Complex Variables, 48 (2003), 429-440.

\bibitem{BaraSaff}  T. Bloom, J. Geronimo, P. Iliev, D. Lubinsky, E. B. Saff, {\em New perspectives in univariate and multivariate orthogonal polynomials}, Report from Banff International Research Station, October 2010.

\bibitem{Carleman} T. Carleman, {\em \"{U}ber die Approximation analytisher Funktionen durch lineare Aggregate von vorgegebenen Potenzen}, Ark. Mat., Astr. Fys. 17 (1923) no. 9, 215--244.

\bibitem{FinGap2}  J. Christiansen, B. Simon, and M. Zinchenko, {\em Finite gap Jacobi matrices II: The Szeg\H{o} class}, Constr. Approx. 33 (2011) 365--403

\bibitem{Duren} P. Duren, {\em Theory of $H^p$ Spaces}, Dover Publications, New York, 1970.

\bibitem{Dur} P. Duren, {\em Univalent Functions} Grundlehren der Mathematischen Wissenschaften; 259, Springer, New York, 1983.

\bibitem{GarnMar} J. Garnett and D. Marshall, {\em Harmonic Measure}, Cambridge University Press, Cambridge, 2005.

\bibitem{Geronimus} Ja. L. Geronimus, {\em Some extremal problems in $L_p(\sigma)$ spaces},
    Math Sbornik 31 (1952) 3--23. [In Russian]

\bibitem{Kalicurve} V. Kaliaguine, {\em On asymptotics of $L_p$ extremal polynomials on a complex curve $0<p<\infty$}, J. Approx. Theory 74 (1993), 226--236.

\bibitem{Kaliarc} V. Kaliaguine, {\em A note on the asymptotics of the orthogonal polynomials on a complex arc: the case of a measure with discrete part.}, J. Approx. Theory 80 (1995), 138--145.

\bibitem{KaliKon} V. Kaliaguine and A. Kononova {\em On the asymptotics of polynomials orthogonal on a system of arcs with respect to a measure with a discrete part.}, St. Petersburg Math. J., 21 (2010), 217--230.

\bibitem{KiSi} R. Killip, B. Simon, {\em  Sum rules for Jacobi matrices and their applications to spectral theory}, Ann. of Math., (2) 158 (2003), no. 1, 253--321.

\bibitem{LubBerg} D. Lubinsky, {\em Universality type limits for Bergman orthogonal polynomials}, Computaional Methods and Function Theory 10 (2010), no. 1, 135--154.

\bibitem{MDCurve}  E. Mi\~{n}a-D\'{i}az, {\em An expansion for polynomials orthogonal over an analytic Jordan curve}, Comm. Math. Phys. 285 (2009), 1109--1128.

\bibitem{MDFaber}  E. Mi\~{n}a-D\'{i}az, {\em On the asymptotic behavior of Faber polynomials for domains with piecewise analytic boundary}, Constr. Approx. 29 (2009), 421--448.

\bibitem{MDPoly} E. Mi\~{n}a-D\'{i}az, {\em Asymptotics for polynomials orthogonal over the unit disk with respect to a positive polynomial weight}, J. Math. Anal. Appl., Vol. 372, no. 1 (2010), 306--315.

\bibitem{MSSzero} E. Mi\~{n}a-D\'{i}az, E. B. Saff, and N. Stylianopoulos, {\em Zero distributions for polynomials orthogonal with weights over certain planar regions}, Computational Methods and Function Theory 5, no. 1 (2005), 185--221.

\bibitem{Koosis} F. Nazarov, A. Volberg, and P. Yuditskii, {\em  Asymptotics of orthogonal polynomials via the Koosis Theorem}, (English Summary) Math. Res. Lett., 13 (2006), no. 5-6, 975--983.

\bibitem{PanLi} X. Li and K. Pan, {\em  Asymptotic behavior of orthogonal polynomials corresponding to measure with discrete part off the unit circle}, J. Approx. Theory, 79 (1994), no. 1, 54--71.

\bibitem{NevaiOP} P. Nevai, {\em Orthogonal polynomials}, Mem. Amer. Math. Soc. 18 (1979), No. 213, 185 pp.

\bibitem{PrYd} F. Peherstorfer and P Yuditskii, {\em Asymptotics of orthogonal polynomials in the presence of a denumerable set of mass points}, Proc. Amer. Math. Soc. 129 (2001), 3213--3220.

\bibitem{Pinkus} A. Pinkus, {\em On $L^1$ Approximation}, Cambridge University Press,
    New York, NY, 1989.

\bibitem{Ransford} T. Ransford, {\em Potential Theory in the Complex Plane}, Cambridge University Press,
    New York, NY, 1995.

\bibitem{Rudin} W. Rudin, {\em  Real and Complex Analysis, Third Edition}, McGraw-Hill, Madison, WI, 1987.

\bibitem{SaffConj} E. B. Saff, {\em Remarks on relative asymptotics for general orthogonal polynomials}, Contemp. Math. Journal, vol. 507, Amer. Math. Soc., Providence, RI (2010), 233--239.

\bibitem{SaffTot} E. B. Saff and V. Totik {\em Logarithmic Potentials with External Fields}, Grundlehren der
    Mathematischen Wissenschaften, Band 316, Springer, Berlin-Heidelberg, 1997.
    
\bibitem{What} E. B. Saff and V. Totik {\em What parts of a measure's support attract zeros of the corresponding orthogonal polynomials}, Proc. Amer. Math. Soc. vol. 114, (1992), 185--190. 

\bibitem{StSh} R. Shakarchi, E. M. Stein, {\em Complex Analysis}, Princeton University Press, Princeton, NJ, 2003.

\bibitem{OPUC1} B. Simon, {\em Orthogonal Polynomials on the Unit Circle, Part One: Classical Theory},
    American Mathematical Society, Providence, RI, 2005.

\bibitem{OPUC2} B. Simon, {\em Orthogonal Polynomials on the Unit Circle, Part Two: Spectral Theory},
    American Mathematical Society, Providence, RI, 2005.

\bibitem{SimPot} B. Simon, {\em Equilibrium measures and capacities in spectral theory}, Inverse Problems and Imaging 1 (2007), 713--772.

\bibitem{WeakCD} B. Simon, {\em Weak convergence of CD kernels and applications}, Duke Math. J. 146 (2009), 305--330.

\bibitem{Rice} B. Simon, {\em Szeg\H{o}'s Theorem and its Descendants: Spectral Theory for $L^2$ perturbations of Orthogonal Polynomials}, Princeton University Press, Princeton, NJ, 2010.

\bibitem{SnipWard} M. Snipes and L. Ward, {\em Convergence properties of harmonic measure distributions for planar domains}, Complex Var. Elliptic Equ., Vol. 53 (2008), 897-913.

\bibitem{StaTo} H. Stahl and V. Totik, {\em General Orthogonal Polynomials}, Cambridge University Press,
    Cambridge, 1992.

\bibitem{Corner} N. Stylianopoulos, {\em Strong asymptotics for Bergman orthogonal polynomials over domains with corners}, preprint (2010).

\bibitem{Suetin} P. K. Suetin, {\em Polynomials Orthogonal Over a Region and Bieberbach Polynomials},
    American Mathematical Society, Providence, RI, 1974.

\bibitem{Szego} G. Szeg\H{o}, {\em Orthogonal Polynomials}, American Mathematical Society Colloquium Publications, Vol. 23, American Mathematical Society, Providence, RI, 1939.

\bibitem{ToTrans} V. Totik, {\em Christoffel functions on curves and domains}, Trans. Amer. Math. Soc. 362 no. 4 (2010), 2053--2087.

\bibitem{Ull1} J. L. Ullman, {\em Orthogonal polynomials for general measures. I: Rational Approximation and Interpolation}, Lecture Notes in Math., 1105, Springer, Berlin (1984), 524--528.

\bibitem{Ull2} J. L. Ullman, {\em Orthogonal polynomials for general measures. II: Orthogonal Polynomials and Applications}, Lecture Notes in Math., 1171, Springer, Berlin (1985), 247--254.

\bibitem{Widom} H. Widom, {\em Extremal polynomials associated with a system of curves and arcs in the complex plane}, Adv. Math. 3 (1969), 127--232.

\bibitem{Wid2} H. Widom, {\em Polynomials associated with measures in the complex plane}, J. Math. Mech 16 (1967), 997--1013.



\end{thebibliography}
\end{document}